\documentclass[pdflatex,sn-mathphys-ay]{sn-jnl}
\usepackage{graphicx}%
\usepackage{multirow}%
\usepackage{amsmath,amssymb,amsfonts}%
\usepackage{amsthm}%
\usepackage{mathrsfs}%
\usepackage[title]{appendix}%
\usepackage{xcolor}%
\usepackage{textcomp}%
\usepackage{manyfoot}%
\usepackage{booktabs}%
\usepackage{algorithm}%
\usepackage{algorithmicx}%
\usepackage{algpseudocode}%
\usepackage{listings}%
\usepackage{stmaryrd}

\usepackage{circuitikz}
\usetikzlibrary{decorations.markings}
\usetikzlibrary{shapes.geometric}

\pgfdeclarelayer{edgelayer}
\pgfdeclarelayer{nodelayer}
\pgfsetlayers{edgelayer,nodelayer,main}
\usepackage{tikz}
\usetikzlibrary{graphs,graphs.standard}
\tikzstyle{none}=[inner sep=0pt]

\tikzstyle{rn}=[circle,fill=Red,draw=Black,line width=0.8 pt]
\tikzstyle{gn}=[circle,fill=Lime,draw=Black,line width=0.8 pt]
\tikzstyle{yn}=[circle,fill=Yellow,draw=Black,line width=0.8 pt]

\tikzstyle{simple}=[-,draw=Black,line width=2.000]
\tikzstyle{arrow}=[-,draw=Black,postaction={decorate},decoration={markings,mark=at position .5 with {\arrow{>}}},line width=2.000]
\tikzstyle{tick}=[-,draw=Black,postaction={decorate},decoration={markings,mark=at position .5 with {\draw (0,-0.1) -- (0,0.1);}},line width=2.000]

\tikzstyle{new style 0}=[draw=black, fill=white, shape=circle, minimum size=8mm]
\tikzstyle{new edge style 0}=[draw=black, fill=white]
\tikzstyle{dashed_border}=[fill=white, draw=black, shape=ellipse, dashed, tikzit shape=ellipse, minimum size=7mm]
\tikzstyle{rotated90}=[rotate=90]
\tikzstyle{rotated45}=[rotate=45]
\tikzstyle{medium_circ}=[fill=white, draw=black, shape=circle, minimum size=9mm]
\tikzstyle{shared_medium_circ}=[fill=white, draw=black, shape=circle, minimum size=9mm, minimum size=17mm, fill=gray, fill opacity=0.5, text opacity=1]
\tikzstyle{rotate-45}=[rotate=-45]
\tikzstyle{arrow}=[->]
\tikzstyle{border}=[-, dashed]
\tikzstyle{f}=[-, draw=none, fill=blue, semitransparent]
\tikzstyle{cooler_arrow}=[{|->}]

\theoremstyle{thmstyleone}%
\newtheorem{theorem}{Theorem}%
\newtheorem{proposition}[theorem]{Proposition}%
\newtheorem{lemma}[theorem]{Lemma}%

\theoremstyle{thmstyletwo}%
\newtheorem{remark}{Remark}%

\theoremstyle{thmstylethree}%
\newtheorem{definition}{Definition}%

\begin{document}

\title[Article Title]{Optimization on the Extended Tensor-Train Manifold with Shared Factors}

\author*[1]{\fnm{Alexander} \sur{Molozhavenko}}\email{amolojavenko@hse.ru}

\author[1]{\fnm{Maxim} \sur{Rakhuba}}

\affil[1]{\orgname{HSE University}}

\abstract{This paper studies tensors that admit decomposition in the Extended Tensor Train (ETT) format, with a key focus on the case where some decomposition factors are constrained to be equal.
This factor sharing introduces additional challenges, as it breaks the multilinear structure of the decomposition.
Nevertheless, we show that Riemannian optimization methods can naturally handle such constraints and prove that the underlying manifold is indeed smooth.
We develop efficient algorithms for key Riemannian optimization components, including a retraction operation based on quasi-optimal approximation in the new format, as well as tangent space projection using automatic differentiation.
Finally, we demonstrate the practical effectiveness of our approach through tensor approximation tasks and multidimensional eigenvalue problem.}



\maketitle

\section{Introduction}\label{sec:intro}
Tensor structures play a fundamental role in modern scientific computing and machine learning applications. 
Among various approaches for handling multidimensional tensors, tensor decompositions have emerged as particularly effective tools. 
A widely used decomposition is the Tensor Train (TT) format \citep{oseledets2011tensor}, which represents a tensor as a chain of smaller three-dimensional tensors called TT-cores. 
An important modification, the Extended Tensor Train (ETT) decomposition, further separates physical indices from the TT-cores into distinct factors. 
From a structural perspective, the ETT representation can be viewed as a special case of the Hierarchical Tucker (HT) decomposition \citep{hackbusch2009new}, with specific constraints on its factorization pattern. 
The ETT decomposition proves especially useful when working with tensors exhibiting large mode sizes, offering enhanced computational efficiency.

This work focuses on a constrained version of the decomposition where we enforce equality between selected factors. We refer to this new structure as the Shared-Factor Extended Tensor Train (SF-ETT) decomposition. We demonstrate empirically that factor sharing enables more compact representations and greater computational efficiency in suitable applications. Beyond symmetric tensors, this approach reduces parameter counts when approximating tensorized (quantized) representations~\citep{approxmattt,khoromskij2011d} and other functions exhibiting similar behavior across modes, outperforming non-shared decompositions.

On the other hand, the factor equality constraint introduces algorithmic challenges. 
Most significantly, it breaks the multilinearity of the tensor format, which is fundamental to many efficient decomposition algorithms. 
For instance, the alternating least squares (ALS) approach~\citep{holtz2012alternating} typically fixes all cores except one, resulting in a linear tensor space. 
This allows for solving standard low-dimensional least squares problems to update individual cores. 
However, this property no longer holds in the shared-factor scenario.

\paragraph{Contributions} To address the aforementioned challenges, we leverage another fundamental property of tensor decompositions -- their structure as smooth manifolds. This geometric perspective enables the application of Riemannian optimization techniques. In particular, in this work, we specifically prove that the shared-factor structure maintains the smooth manifold property -- that is, the set of fixed-rank SF-ETT tensors constitutes a smooth manifold.

We derive efficient algorithms for the essential components of Riemannian optimization in the SF-ETT framework. Specifically, we develop the SF-ETT-SVD algorithm, which constructs  approximations in the SF-ETT format through an efficient SVD-based procedure and is applicable for any number of shared factors. We prove that these approximations satisfy quasi-optimality bounds, making them suitable for use as retractions in Riemannian optimization.
Another crucial component is the tangent space projection, an essential operation for computing Riemannian gradients and implementing vector transport. We derive efficient projection formulas and demonstrate how to implement it within automatic differentiation frameworks. This approach ensures both generality and broad applicability across various functionals and practical scenarios.

The resulting framework is implemented as an open-source \texttt{pytorch} \href{https://github.com/SuperCrabLover/SF-ETT-Manifold.git}{package} that naturally supports GPU parallelization and automatic differentiation. The framework is tested on several examples, including function approximation and multidimensional eigenvalue problem.

\paragraph{Outline}

Section~\ref{sec:notation} introduces the essential tensor notation and formats that are used later for the SF-ETT decomposition.
In Section~\ref{sub:sfett_format}, we introduce the SF-ETT decomposition and define the concept of SF-ETT rank. Building on this, Section~\ref{sub:sfett_svd_rounding} presents Algorithm~\ref{alg:sfett_rounding}, which constructs a quasi-optimal SF-ETT approximation via matrix truncations.

To enable efficient optimization over SF-ETT representations, we develop a Riemannian framework. Namely, Section~\ref{sub:sfett_manifold_smoothness} characterizes the smooth manifold of tensors with fixed SF-ETT rank. Section~\ref{sub:sfett_tangent_space} derives the associated tangent space structure, while Section~\ref{sub:sfett_tangent_space_proj} provides explicit formulas for the projection operator onto this space. Section~\ref{sec:gradient_based_optimization} consolidates these components into a Riemannian optimization framework, including closed-form expressions for the Riemannian gradient, numerically stable retraction and an automatic differentiation (autodiff)-based implementation for calculation of vector transport and Riemannian gradient.

Finally, in Section~\ref{sec:num_exp}, we demonstrate the efficacy of our framework through numerical experiments. Section~\ref{sub:grid_functions} presents its application to the approximation of functions on uniform grids, and Section~\ref{sub:eigenvalue_problem} showcases its use for solving eigenvalue problems.

\section{Related work}\label{sec:rel_work}
    The smooth manifold structure is inherent in various tensor decompositions, such as the Tucker decomposition~\citep{tucker1963implications},  TT~\citep{holtz2012manifolds}, and the more general tree-based HT tensor format~\citep{uschmajew2013geometry}.  
    However, this property does not hold for the CP decomposition~\citep{hitchcock1927expression} or, more generally, for tensor networks containing loops~\citep{gao2024riemannian} such as the Tensor Chain (Tensor Ring) format~\citep{khoromskij2011d,zhao2016tensor}. For more details concerning tensor formats see \citep{KHOROMSKIJ20121,kolda2009tensor,grasedyck2013literature}.
    For tensor formats that lead to smooth structures one may use the Riemannian optimization techniques~\citep{absil2008optimization,boumal2023introduction}, which lead to efficient and potentially parallelizable methods.
    
    Tensor-based Riemannian optimization approaches proved useful in different applications.
    For example, they have been used for matrix and tensor completion~\citep{kressner2014low}, linear systems~\citep{kressner2016preconditioned} as well as for the multidimensional eigenvalue problems~\citep{rakhuba2019low}.
    Low-rank Riemannian optimization has also recently emerged in machine learning applications, see~\citep{peshekhonov2024training} for factorization models and~\citep{bogachev2025riemannlora} for low-rank adaptation in large language models.
    The need for Riemannian optimization frameworks that support GPU along with automatic differentiation has led to the development of \texttt{t3f}, \texttt{ttax}, \texttt{tntorch} software packages \citep{ttax2020, t3f2017, tntorch}.
    Classic python implementation of Riemannian tools is also available in the general-purpose TT package \texttt{ttpy}~\citep{ttpy2012}.
    
    The idea of sharing factors was considered in the RESCAL model~\citep{nickel2011three}, a relaxed variant of DEDICOM~\citep{harshman1978models}, which enforces equal factors in a Tucker-2 decomposition through a heuristic approach.
	We also mention the work \citep{obukhov2020t}, which
    used shared basis for TT cores and \citep{peshekhonov2024training}, which considered shared factors in the Tucker decomposition context. 
	The authors of \citep{balavzevic2019tucker} seek for 
	the solution of the knowledge graph completion problem with shared factors and the gradient-based approach. 
    In all these works, sharedness has demonstrated significant utility in reducing the total number of parameters while effectively extracting key features from modes with analogous semantic or structural roles.

\section{Preliminaries}\label{sec:notation}
	The primary object of consideration in this work is a concept of tensor as multidimensional  
	array of real numbers. Let $ \mathcal{X} \in \mathbb{R}^{n_1\times \dotsc\times n_d}$ be a $d$-dimensional 
	tensor then by
	\begin{equation*}
		\mathcal{X}_{i_1, \dotsc, i_d} = \mathcal{X}(i_1, \dotsc, i_d) \in \mathbb{R}, \quad 
		i_j \in \overline{1, n_j}, \quad j \in \overline{1, d},
	\end{equation*}
	we denote its element in the position $i_1, \dotsc, i_d$.
    The number of elements in a tensor grows exponentially with its dimension $d$, motivating the use of tensor decompositions to achieve substantial memory reduction.
    The Frobenius tensor norm  and the scalar tensor product are defined as follows 
	\begin{equation*}
		\|\mathcal{X}\|^2_F = \sum_{i_1, \dots, i_d}^{n_1, \dots, n_d} \mathcal{X}_{i_1, \dots, i_d}^2, \quad
		\langle \mathcal{X}, \mathcal{Y} \rangle = \sum_{i_1, \dots, i_d}^{n_1, \dots, n_d} 
		\mathcal{X}_{i_1, \dots, i_d}\mathcal{Y}_{i_1, \dots, i_d}.
	\end{equation*}
	The Kronecker product is denoted as $\otimes_K$, the Moore-Penrose pseudo-inverse of matrix~$A$ is denoted as $A^\dagger$, and
	the Hadamard product (elementwise) is denoted as~$\odot$.
    
    Before presenting specific tensor formats, we first discuss various methods for reorganizing tensor elements into matrix form, as the properties of tensor decompositions fundamentally depend on these matrix representations.

\subsection{Tensor matricizations and unfoldings}
	Let us consider two ways of transforming a tensor into a matrix which shed light on multiple
	properties of  tensor decompositions. The $\mu$-th matricization $\mathcal{X}_{(\mu)} \in \mathbb{R}^{n_\mu \times \prod n_j / n_\mu }$ \citep{steinlechner2016riemannian} is a matrix with elements
	\begin{equation}\label{eq:matricization}
		\left( {\mathcal{X}_{(\mu)}} \right)_{i_\mu, i_{\neq \mu}} =  \mathcal{X}_{i_1, \dotsc, i_\mu, \dotsc, i_d},
	\end{equation}
	where
	\begin{equation*}
		i_{\neq \mu} = \sum_{\nu = 1, \nu\neq \mu}^{d}\left(  i_\nu \prod_{\tau =1, \tau \neq\mu}^{\nu - 1} n_\tau\right).
	\end{equation*}
	This way of mapping a tensor into a matrix is used in the concept of Tucker format defined below.
	Another way is a construction of the $\mu$-th unfolding of a tensor $\mathcal{X}$ denoted as
	$\mathcal{X}^{<\mu>} \in \mathbb{R}^{(n_1n_2 \dots n_\mu) \times (n_{\mu +1} \dots n_d)}$ 
	\citep{steinlechner2016riemannian}. The $\mu$-th unfolding is a matrix with elements
	\begin{equation}\label{eq:unfolding}
		\left( {\mathcal{X}^{<\mu>}} \right)_{i_{\text{row}}, i_{\text{col}}} =  \mathcal{X}_{i_1, \dotsc, i_d},
	\end{equation}
	where
	\begin{equation*}
		i_{\text{row}} = \sum_{\nu = 1}^{\mu}\left(i_\nu \prod_{\tau =1}^{\nu - 1} n_\tau \right), \quad
		i_{\text{col}} = \sum_{\nu = \mu + 1}^{d}\left(i_\nu \prod_{\tau =\mu + 1}^{\nu - 1} n_\tau\right).
	\end{equation*}
	This transformation is connected to the concept of the TT format that are also described below.
	Working with tensors in such a format implies working with $3$-dimensional tensors (TT-cores), so it is convenient 
	to use specific notation for them. Let $W \in \mathbb{R}^{n_1\times n_2\times n_3}$ be a $3$-dimensional tensor.
	Then the unfoldings $W^{<2>}$ and $W^{<1>}$ are called respectively left and right unfoldings of $W$:
	denoted as  
	\begin{equation*}
		L(W) = W^{<2>} \in \mathbb{R}^{n_1 n_2 \times n_3}, \quad R(W) = W^{<1>} \in \mathbb{R}^{n_1 \times n_2 n_3}.
	\end{equation*}
	The tensor $W$ is called left or right orthogonal if respectively
	\begin{equation*}
		L(W)^\top L(W) = I_{n_3}, \quad\text{or}\quad R(W)R(W)^\top = I_{n_1}.
	\end{equation*}

\subsection{Tensor formats}%
\label{sec:tensor_formats}
	This section is devoted to the definition of tensors format that we ues in the paper and their basic properties. 
    Tensor formats have 
	simple and useful graphical representations that are called tensor diagrams, see~Figure~\ref{fig:tens_diags}.
    A node with $d$ legs represents a $d$-dimensional tensor.
    Contraction of two tensors along an index is represented by two nodes with a common leg.
    To also denote orthogonality, we use half-filled nodes: contracting a tensor with itself along nodes from a white-filled region must produce the identity tensor.
	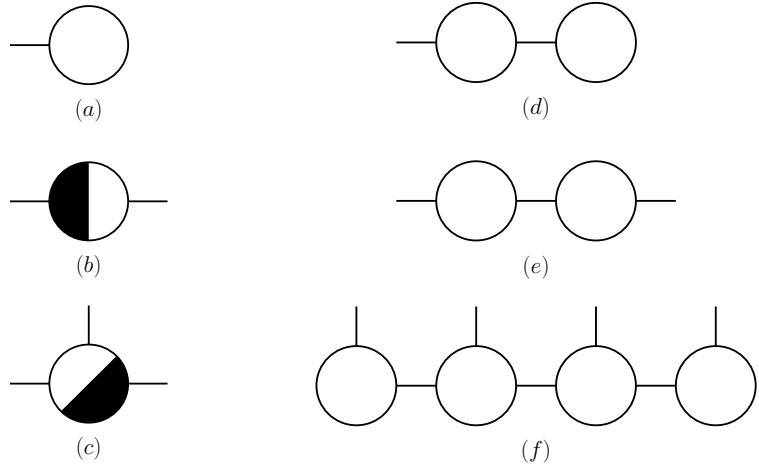
\begin{figure}[h!tp]
		\centering
		\begin{minipage}{.45\textwidth}
			\centering
			\resizebox{0.37\textwidth}{!}{
                \begin{circuitikz}
                	\draw [line width=2pt] (11.5,16.5) circle (1.5cm);
                	\draw [line width=2pt] (8.5,16.5) .. controls (9.25,16.5) and (9.25,16.5) .. (10,16.5);
                	\node [font=\Huge] at (11.5,14) {$(a)$};
                
                	\draw [line width=2pt] (11.5,10.5) circle (1.5cm);
                	\draw [line width=2pt] (8.5,10.5) .. controls (9.25,10.5) and (9.25,10.5) .. (10,10.5);
                	\draw [line width=2pt] (13,10.5) .. controls (13.75,10.5) and (13.75,10.5) .. (14.5,10.5);
                	\node [font=\Huge] at (11.5,8) {$(b)$};
                	\fill [black] (11.5,10.5) -- ++(270:1.5cm) arc (270:90:1.5cm) -- cycle;

                	\draw [line width=2pt] (11.5,3.5) circle (1.5cm);
                	\draw [line width=2pt] (10,3.5) .. controls (9.25,3.5) and (9.25,3.5) .. (8.5,3.5);
                	\draw [line width=2pt] (11.5,5) .. controls (11.5,5.75) and (11.5,5.75) .. (11.5,6.5);
                	\draw [line width=2pt] (13,3.5) .. controls (13.75,3.5) and (13.75,3.5) .. (14.5,3.5);
                	\node [font=\Huge] at (11.5,1) {$(c)$};
                	\fill [black] (11.5,3.5) -- ++(45:1.5cm) arc (45:-135:1.5cm) -- cycle;
                \end{circuitikz}
            }%
		\end{minipage}%
		\begin{minipage}{0.45\textwidth}
			\centering
			\resizebox{1\textwidth}{!}{
                \begin{circuitikz}
                	\draw [line width=2pt] (9.25,16.5) circle (1.5cm);
                	\draw [line width=2pt] (13.75,16.5) circle (1.5cm);
                	\draw [line width=2pt] (6.25,16.5) .. controls (7,16.5) and (7,16.5) .. (7.75,16.5);
                	\draw [line width=2pt] (10.75,16.5) .. controls (11.5,16.5) and (11.5,16.5) .. (12.25,16.5);
                	\node [font=\Huge] at (11.5,14) {$(d)$};
                
                	\draw [line width=2pt] (9.25,10.5) circle (1.5cm);
                	\draw [line width=2pt] (13.75,10.5) circle (1.5cm);
                	\draw [line width=2pt] (6.25,10.5) .. controls (7,10.5) and (7,10.5) .. (7.75,10.5);
                	\draw [line width=2pt] (10.75,10.5) .. controls (11.5,10.5) and (11.5,10.5) .. (12.25,10.5);
                	\draw [line width=2pt] (15.25,10.5) .. controls (16,10.5) and (16,10.5) .. (16.75,10.5);
                	\node [font=\Huge] at (11.5,8) {$(e)$};
                
                	\draw [line width=2pt] (4.75,3.5) circle (1.5cm);
                	\draw [line width=2pt] (4.75,5) .. controls (4.75,5.75) and (4.75,5.75) .. (4.75,6.5);
                	\draw [line width=2pt] (6.25,3.5) .. controls (7,3.5) and (7,3.5) .. (7.75,3.5);
                	\draw [line width=2pt] (9.25,3.5) circle (1.5cm);
                	\draw [line width=2pt] (9.25,5) .. controls (9.25,5.75) and (9.25,5.75) .. (9.25,6.5);
                	\draw [line width=2pt] (10.75,3.5) .. controls (11.5,3.5) and (11.5,3.5) .. (12.25,3.5);
                	\draw [line width=2pt] (13.75,3.5) circle (1.5cm);
                	\draw [line width=2pt] (13.75,5) .. controls (13.75,5.75) and (13.75,5.75) .. (13.75,6.5);
                	\draw [line width=2pt] (15.25,3.5) .. controls (16,3.5) and (16,3.5) .. (16.75,3.5);
                	\draw [line width=2pt] (18.25,3.5) circle (1.5cm);
                	\draw [line width=2pt] (18.25,5) .. controls (18.25,5.75) and (18.25,5.75) .. (18.25,6.5);
                	\node [font=\Huge] at (11.5,1) {$(f)$};
                \end{circuitikz}
            }%
		\end{minipage}
		\caption{Tensor diagrams: $(a)$ --- a vector, $(b)$ --- a matrix with orthonormal columns,
			$(c)$~---~ a left-orthogonal tensor,
		$(d)$ --- a matrix-vector product, $(e)$~---~a matrix-matrix product, $(f)$ --- the TT decomposition}
		\label{fig:tens_diags}
	\end{figure}
	\subsubsection{Tucker format}\label{sub:tucker_format}
		The Tucker format~\citep{tucker1963implications} of a $d$-dimensional  
		tensor $ \mathcal{X} \in \mathbb{R}^{n_1\times \dotsc\times n_d}$ with rank parameters
		$(r^t_1, \dotsc, r^t_d)$ is its representation as a contraction 
		of $d$ Tucker factors $U^{(i)} \in \mathbb{R}^{n_i \times r^t_i}, i \in \overline{1,d}$ 
		with a Tucker core $\mathcal{G} \in \mathbb{R}^{r^t_1\times \dotsc\times r^t_d}$:
		\begin{equation}\label{eq:tucker_dec}
			\mathcal{X}_{i_1, \dotsc, i_d}	= 
			\sum_{j_1, \dotsc, j_d}^{r^t_1, \dotsc, r^t_d} 
			G_{j_1, \dotsc, j_d}U^{(1)}_{i_1, j_1}\dotsc U^{(d)}_{i_d, j_d},
		\end{equation}
		see the corresponding diagram in Figure~\ref{fig:tucker_format}.
		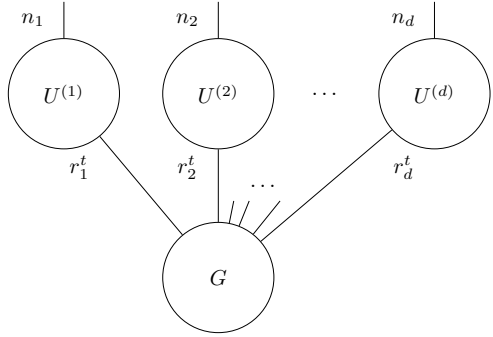
\begin{figure}[h!tp]
			\centering
			\resizebox{0.5\textwidth}{!}{
                \begin{tikzpicture}
                	\begin{pgfonlayer}{nodelayer}
                		\node [style={medium_circ}, minimum size=18mm] (0) at (-1.5, -2) {$G$};
                		\node [style={medium_circ}, minimum size=18mm] (1) at (-4, 1) {$U^{(1)}$};
                		\node [style={medium_circ}, minimum size=18mm] (2) at (-1.5, 1) {$U^{(2)}$};
                		\node [style={medium_circ}, minimum size=18mm] (4) at (2, 1) {$U^{(d)}$};
                		\node [style=none] (8) at (-4, 2.5) {};
                		\node [style=none] (9) at (-1.5, 2.5) {};
                		\node [style=none] (13) at (2, 2.5) {};
                		\node [style=none] (14) at (0.25, 1) {$\dotsc$};
                		\node [style=none] (15) at (-2, -0.20) {$r^t_2$};
                		\node [style=none] (16) at (-3.75, -0.20) {$r^t_1$};
                		\node [style=none] (18) at (1.5, -0.2) {$r^t_d$};
                		\node [style=none] (19) at (-4.5, 2.2) {$n_1$};
                		\node [style=none] (20) at (-2, 2.2) {$n_2$};
                		\node [style=none] (22) at (1.5, 2.2) {$n_{d}$};
                		\node [style=none] (23) at (-1, -0.75) {};
                		\node [style=none] (24) at (-0.75, -0.5) {$\dotsc$};
                		\node [style=none] (25) at (-1.25, -0.75) {};
                		\node [style=none] (27) at (-0.5, -0.75) {};
                	\end{pgfonlayer}
                	\begin{pgfonlayer}{edgelayer}
                		\draw (0) to (1);
                		\draw (0) to (2);
                		\draw (0) to (4);
                		\draw (1) to (8.center);
                		\draw (2) to (9.center);
                		\draw (13.center) to (4);
                		\draw (0) to (23.center);
                		\draw (0) to (25.center);
                		\draw (0) to (27.center);
                	\end{pgfonlayer}
                \end{tikzpicture}
            }
			\caption{Tucker decomposition as a tensor diagram}
			\label{fig:tucker_format}
		\end{figure}

		There exists a tuple of parameters $(r^t_1, \dotsc, r^t_{d})$ with simultaneously attained minimal possible 
		$r^t_i$. Such tuple is called the Tucker rank of a tensor $\mathcal{X}$ and the minimal 
		$r^t_i$ can be represented as a rank of the corresponding matricization 
		$\mathcal{X}_{(i)}$ of the tensor $\mathcal{X}$ defined in \eqref{eq:matricization}
		\begin{equation*}
			r^t_i = \mathrm{rank}\,(\mathcal{X}_{(i)}), \quad i \in \overline{1, d}.
		\end{equation*}
		The Tucker decomposition~\eqref{eq:tucker_dec} is further formally denoted as
		\begin{equation}\label{eq:tucker_dec_cool_notation}
			\mathcal{X}	= \left\llbracket \mathcal{G}; U^{(1)}, \dotsc, U^{(d)}\right\rrbracket.
		\end{equation}
		Note that the dimension of the Tucker core equals to the dimension of
		the tensor. As a result, this format suffers from the exponential growth of the number of parameters -- manifestation of the the curse of dimensionality.

	\subsubsection{Tensor Train format}\label{sub:tt_format}
		The TT decomposition~\citep{oseledets2011tensor,oseledets2009breaking}   of 
		a tensor $ \mathcal{X} \in \mathbb{R}^{n_1\times \dotsc\times n_d}$ with the rank
		parameters $(r^{tt}_1, \dotsc, r^{tt}_{d-1})$  
		is defined as:
		\begin{equation}\label{eq:tt_dec}
			 \mathcal{X}_{i_1, \dotsc, i_d}	=  
			  W^{(1)}(i_1) W^{(2)}(i_2)\dotsc W^{(d)}(i_d),
		\end{equation}
		where 
		\begin{equation*}
			W^{(j)}(i_j) \in \mathbb{R}^{r_{j - 1}^{tt} \times r^{tt}_j}, \quad i_j \in \overline{1, n_j}, 
			\quad j \in \overline{1, d},\quad
			r^{tt}_0 = r^{tt}_{d} = 1.
		\end{equation*}
		The corresponding tensor diagram is presented in 
		Figure~\ref{fig:tt_format}.
		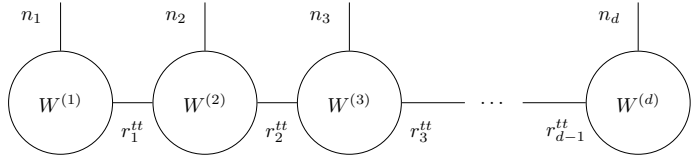
\begin{figure}[h!tp]
			\centering
			\resizebox{0.7\textwidth}{!}{
                \begin{tikzpicture}
                	\begin{pgfonlayer}{nodelayer}
                		\node [style={medium_circ}, minimum size=18mm] (0) at (-6, -0.5) {$W^{(1)}$};
                		\node [style={medium_circ}, minimum size=18mm] (1) at (-3.5, -0.5) {$W^{(2)}$};
                		\node [style={medium_circ}, minimum size=18mm] (2) at (-1, -0.5) {$W^{(3)}$};
                		\node [style=none] (3) at (1.5, -0.5) {$\dotsc$};
                		\node [style={medium_circ}, minimum size=18mm] (5) at (4, -0.5) {$W^{(d)}$};
                		\node [style=none] (6) at (1, -0.5) {};
                		\node [style=none] (7) at (2, -0.5) {};
                		\node [style=none] (8) at (-6, 1.25) {};
                		\node [style=none] (9) at (-3.5, 1.25) {};
                		\node [style=none] (10) at (-1, 1.25) {};
                		\node [style=none] (12) at (4, 1.25) {};
                		\node [style=none] (13) at (-6.5, 1) {$n_1$};
                		\node [style=none] (14) at (-4, 1) {$n_2$};
                		\node [style=none] (15) at (-1.5, 1) {$n_3$};
                		\node [style=none] (17) at (3.5, 1) {$n_d$};
                		\node [style=none] (18) at (-4.75, -1) {$r^{tt}_1$};
                		\node [style=none] (19) at (-2.25, -1) {$r^{tt}_2$};
                		\node [style=none] (20) at (0.25, -1) {$r^{tt}_3$};
                		\node [style=none] (21) at (2.75, -1) {$r^{tt}_{d-1}$};
                	\end{pgfonlayer}
                	\begin{pgfonlayer}{edgelayer}
                		\draw (0) to (1);
                		\draw (1) to (2);
                		\draw (2) to (6.center);
                		\draw (5) to (12.center);
                		\draw (2) to (10.center);
                		\draw (1) to (9.center);
                		\draw (0) to (8.center);
                		\draw (5) to (7.center);
                	\end{pgfonlayer}
                \end{tikzpicture}
            }
			\caption{TT decomposition as a tensor diagram}
			\label{fig:tt_format}
		\end{figure}

		There exists a tuple of rank parameters $(r^{tt}_1, \dotsc, r^{tt}_{d - 1})$ with simultaneously attainable minimal possible $r^{tt}_i$. Such a
		tuple is called the TT-rank of a tensor $\mathcal{X}$. The minimal possible values $r^{tt}_i$ can be expressed as a 
		rank of corresponding unfolding $\mathcal{X}^{<i>}$ of tensor $\mathcal{X}$ defined in \eqref{eq:unfolding}:
		\begin{equation*}
			r^{tt}_i = \mathrm{rank}\,\left(\mathcal{X}^{<i>}\right), \quad i \in \overline{1, d - 1}.
		\end{equation*}
		The TT decomposition will be further denoted as
		\begin{equation}\label{eq:tt_dec_cool_notation}
			\mathcal{X} = \mathcal{T}\left(W^{(1)}, \dotsc, W^{(d)}\right).
		\end{equation}
		For the TT decomposition, one can define auxiliary useful tensors \citep{lubich2015time}: 
		\begin{align}
			\mathcal{X}_{\le \mu} &= \mathcal{T}\left(W^{(1)}, \dotsc, W^{(\mu)}\right)^{<\mu>},\quad  
			\mathcal{X}_{\le \mu}
			\in \mathbb{R}^{n_1 \dotsm n_\mu \times r_\mu}\label{eq:unf_less_mu}, \\
			\mathcal{X}_{\ge \mu} &= {\mathcal{T}\left(W^{(\mu)},\dotsc,W^{(d)}\right)^{<1>}}^{\top}, 
			\quad \mathcal{X}_{\ge \mu}
			\in \mathbb{R}^{n_\mu \dotsm n_d \times r_{\mu - 1}}\label{eq:unf_more_mu}, \\
			\mathcal{X}_{\neq \mu} &= \mathcal{X}_{\ge \mu + 1} \otimes_K I_{n_\mu}  
			\otimes_K\mathcal{X}_{\le \mu - 1}\nonumber.
		\end{align}
		In contrast to the Tucker decomposition, when the ranks are bounded, we observe that the TT format has polynomial growth of the total number of parameters, breaking
		 the curse of dimensionality.
		
	\subsubsection{Extended Tensor Train format}\label{sub:extended_tt_format}
		The ETT decomposition format of a tensor 
		$\mathcal{X} \in \mathbb{R}^{n_1\times \dotsc\times n_d}$ with the rank
		parameters $(r^{tt}_1, \dotsc, r^{tt}_{d-1}, r^{t}_1 \dotsc, r^{t}_d)$ can be viewed as a combination the Tucker and the ETT formats.
		In particular, it is a contraction of a Tucker core represented in  the TT format with $d$ Tucker factors
		$U^{(i)} \in \mathbb{R}^{n_i \times r_i^t}$:
		\begin{equation}\label{eq:notation_ett_format}
			\mathcal{X} = \left\llbracket \mathcal{T} \left(W^{(1)}, \dotsc, W^{(d)}\right); 
			U^{(1)}, \dotsc, U^{(d)}\right\rrbracket.
		\end{equation}
		One can also notice that The ETT format is a specific case of the HT 
		decomposition \citep{uschmajew2013geometry}. 
		The tensor diagram for \eqref{eq:notation_ett_format} is presented in  
		Figure~\ref{fig:ett_format}.
		\begin{figure}[h!tp]
			\centering
			\resizebox{0.7\textwidth}{!}{
                \begin{tikzpicture}
                	\begin{pgfonlayer}{nodelayer}
                		\node [style={medium_circ}, minimum size=15mm] (0) at (-6, -0.5) {$W^{(1)}$};
                		\node [style={medium_circ}, minimum size=15mm] (1) at (-3.5, -0.5) {$W^{(2)}$};
                		\node [style={medium_circ}, minimum size=15mm] (2) at (-1, -0.5) {$W^{(3)}$};
                		\node [style=none] (3) at (1.5, -0.5) {$\dotsc$};
                		\node [style={medium_circ}, minimum size=15mm] (5) at (4, -0.5) {$W^{(d)}$};
                		\node [style=none] (6) at (1, -0.5) {};
                		\node [style=none] (7) at (2, -0.5) {};
                		\node [style=none] (8) at (-6, 2) {};
                		\node [style=none] (9) at (-3.5, 2) {};
                		\node [style=none] (10) at (-1, 2) {};
                		\node [style=none] (12) at (4, 2) {};
                		\node [style=none] (13) at (-6.5, 0.75) {$r^{t}_1$};
                		\node [style=none] (14) at (-4, 0.75) {$r^{t}_2$};
                		\node [style=none] (15) at (-1.5, 0.75) {$r^{t}_3$};
                		\node [style=none] (18) at (-4.75, -1) {$r^{tt}_1$};
                		\node [style=none] (19) at (-2.25, -1) {$r^{tt}_2$};
                		\node [style=none] (20) at (0.25, -1) {$r^{tt}_3$};
                		\node [style=none] (21) at (2.75, -1) {$r^{tt}_{d-1}$};
                		\node [style={medium_circ}, minimum size=15mm] (22) at (-6, 2) {$U^{(1)}$};
                		\node [style=none] (23) at (3.5, 0.75) {$r^{t}_d$};
                		\node [style={medium_circ}, minimum size=15mm] (24) at (-3.5, 2) {$U^{(2)}$};
                		\node [style={medium_circ}, minimum size=15mm] (25) at (-1, 2) {$U^{(3)}$};
                		\node [style={medium_circ}, minimum size=15mm] (26) at (4, 2) {$U^{(d)}$};
                		\node [style=none] (27) at (-6, 3.75) {};
                		\node [style=none] (28) at (-3.5, 3.75) {};
                		\node [style=none] (30) at (-1, 3.75) {};
                		\node [style=none] (31) at (4, 3.75) {};
                		\node [style=none] (32) at (-6.5, 3.25) {$n_1$};
                		\node [style=none] (34) at (-4, 3.25) {$n_2$};
                		\node [style=none] (35) at (-1.5, 3.25) {$n_3$};
                		\node [style=none] (36) at (3.5, 3.25) {$n_d$};
                	\end{pgfonlayer}
                	\begin{pgfonlayer}{edgelayer}
                		\draw (0) to (1);
                		\draw (1) to (2);
                		\draw (2) to (6.center);
                		\draw (5) to (12.center);
                		\draw (2) to (10.center);
                		\draw (1) to (9.center);
                		\draw (0) to (8.center);
                		\draw (5) to (7.center);
                		\draw (22) to (27.center);
                		\draw (24) to (28.center);
                		\draw (25) to (30.center);
                		\draw (26) to (31.center);
                	\end{pgfonlayer}
                \end{tikzpicture}
            }
			\caption{ETT decomposition as a tensor diagram}
			\label{fig:ett_format}
		\end{figure}
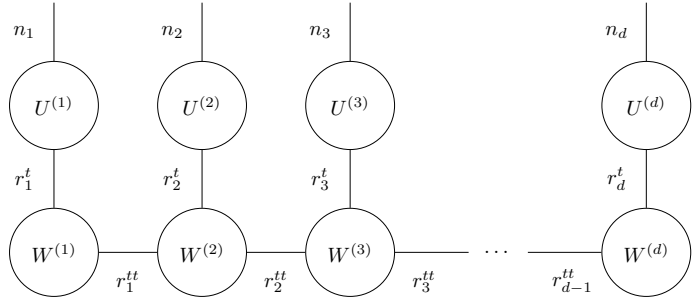

		Similarly, the rank values $(r^{tt}_1, \dotsc, r^{tt}_{d-1}, r^{t}_1 \dotsc, r^{t}_d)$ with the minimal possible values are:
		\begin{align*}
			& r^{tt}_i = \mathrm{rank}\,\left( A^{<i>} \right), \quad i \in \overline{1, d - 1}, \\
			& r^{t}_j = \mathrm{rank}\,\left( A_{(j)} \right), \quad j \in \overline{1, d}.
		\end{align*}

	\subsubsection{Shared Factor Tucker decomposition}\label{sub:sftucker_format}
		The SF-Tucker decomposition \citep{peshekhonov2024training} of a $d$-dimensional ($d = d_t + d_s$)
		tensor $ \mathcal{X} \in \mathbb{R}^{n_1\times \dotsc\times n_{d_t}\times n_{s}\times \dotsc\times n_s}$ 
		with the rank parameters 
		$(r^t_1, \dotsc, r^t_{d_t}, r^t_{s})$ is its representation as a contraction of $d_t$ Tucker factors 
		$U^{(i)} \in \mathbb{R}^{n_i \times r^t_i}, i \in \overline{1,d_t}$, one shared Tucker factor 
		$U\in \mathbb{R}^{n_s \times r^t_s}$ 
		(repeated $d_s$ times) and the Tucker core 
		$\mathcal{G}\in \mathbb{R}^{r^t_1\times \dotsc\times r^t_{d_t}\times r^t_{s}\times \dotsc\times r^t_s}$:
		\begin{equation}\label{eq:sftucker_format_cool_notation}
			\mathcal{X}	= \left\llbracket \mathcal{G}; U^{(1)}, \dotsc, U^{(d_t)}, U, \dotsc, U\right\rrbracket,
		\end{equation}
		the corresponding tensor diagram is in 
		Figure~\ref{fig:sftucker_format}.
		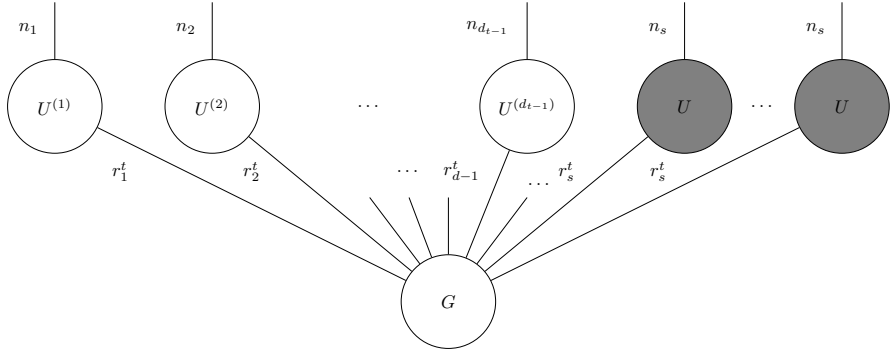
\begin{figure}[htpb]
			\centering
			\resizebox{0.9\textwidth}{!}{
                \begin{tikzpicture}
                	\begin{pgfonlayer}{nodelayer}
                		\node [style={medium_circ}, minimum size=18mm] (0) at (3.5, -2.75) {$G$};
                		\node [style={medium_circ}, minimum size=18mm] (1) at (-4, 1) {$U^{(1)}$};
                		\node [style={medium_circ}, minimum size=18mm] (2) at (-1, 1) {$U^{(2)}$};
                		\node [style={medium_circ}, minimum size=18mm] (4) at (5, 1) {$U^{(d_{t-1})}$};
                		\node [style=none] (8) at (-4, 3) {};
                		\node [style=none] (9) at (-1, 3) {};
                		\node [style=none] (13) at (5, 3) {};
                		\node [style=none] (14) at (2, 1) {$\dotsc$};
                		\node [style=none] (15) at (-0.25, -0.25) {$r^t_2$};
                		\node [style=none] (16) at (-2.75, -0.25) {$r^t_1$};
                		\node [style=none] (18) at (3.75, -0.25) {$r^t_{d-1}$};
                		\node [style=none] (19) at (-4.5, 2.5) {$n_1$};
                		\node [style=none] (20) at (-1.5, 2.5) {$n_2$};
                		\node [style=none] (22) at (4.25, 2.5) {$n_{d_{t-1}}$};
                		\node [style=none] (23) at (2, -0.75) {};
                		\node [style=none] (24) at (2.75, -0.25) {$\dotsc$};
                		\node [style=none] (25) at (2.75, -0.75) {};
                		\node [style=none] (27) at (3.5, -0.75) {};
                		\node [style={shared_medium_circ}, minimum size=18mm] (28) at (11, 1) {$U$};
                		\node [style=none] (29) at (11, 3) {};
                		\node [style={shared_medium_circ}, minimum size=18mm] (31) at (8, 1) {$U$};
                		\node [style=none] (32) at (8, 3) {};
                		\node [style=none] (33) at (7.5, 2.5) {$n_s$};
                		\node [style=none] (34) at (10.5, 2.5) {$n_s$};
                		\node [style=none] (35) at (7.5, -0.25) {$r^t_s$};
                		\node [style=none] (36) at (5.75, -0.25) {$r^t_s$};
                		\node [style=none] (37) at (9.5, 1) {$\dotsc$};
                		\node [style=none] (38) at (5, -0.75) {};
                		\node [style=none] (39) at (5.25, -0.5) {$\dotsc$};
                	\end{pgfonlayer}
                	\begin{pgfonlayer}{edgelayer}
                		\draw (0) to (1);
                		\draw (0) to (2);
                		\draw (0) to (4);
                		\draw (1) to (8.center);
                		\draw (2) to (9.center);
                		\draw (13.center) to (4);
                		\draw (0) to (23.center);
                		\draw (0) to (25.center);
                		\draw (0) to (27.center);
                		\draw (29.center) to (28);
                		\draw (32.center) to (31);
                		\draw (0) to (31);
                		\draw (0) to (28);
                		\draw (0) to (38.center);
                	\end{pgfonlayer}
                \end{tikzpicture}
            }
			\caption{SF-Tucker decomposition as a tensor diagram}
			\label{fig:sftucker_format}
		\end{figure}

		The rank is defined similarly to the Tucker case, but with the additional constraint of equality for some of the modes. As a result, we have a tuple of $d_t$ rank parameters and the additional shared one: $(r^{t}_1, \dotsc, r^{t}_{d_t}, r^{t}_s)$ with minimal $r^{t}_i$. The ranks
		$r^{t}_i, i \in \overline{1, d_t}$ coincide with the respective Tucker ranks, while the shared rank $r^{t}_s$ can be found as follows:
		\begin{equation}\label{eq:shared_sftucker_rank}
			r^{t}_s = \mathrm{rank}\,
			\left( \mathcal{X}_{(d_{t} + 1)} | \mathcal{X}_{(d_{t} + 2)} | \dots | \mathcal{X}_{(d)} \right).
		\end{equation}

\section{Shared-Factor Extended Tensor Train decomposition}\label{sub:sfett_format}
In this section, we present the new SF-ETT decomposition.
		For a $d$-dimensional tensor 
		$\mathcal{X} \in \mathbb{R}^{n_1\times \dotsc\times n_{d_t}\times n_{s}\times \dotsc\times n_s}$, its SF-ETT decomposition reads as:
		\begin{equation}\label{eq:sfett_format}
			\mathcal{X}	= \left\llbracket 
			\mathcal{T}\left(W^{(1)}, \dotsc, W^{(d)}\right); U^{(1)}, \dotsc, U^{(d_t)}, U, \dotsc, U
			\right\rrbracket,
		\end{equation}
		where $W^{(i)} \in \mathbb{R}^{r^{\mathrm{tt}}_{i - 1}\times r^{\mathrm{t}}_{i}\times r^{\mathrm{tt}}_{i}}$, 
		$i \in \overline{1, d_t}$, $W^{(i)} \in 
		\mathbb{R}^{r^{\mathrm{tt}}_{i - 1}\times r^{\mathrm{t}}_{s}\times r^{\mathrm{tt}}_{i}}$,
		$i \in \overline{d_t + 1, d}$
		($r^{\mathrm{tt}}_{0} = r^{\mathrm{tt}}_{d} = 1$), 
		$U^{(j)} \in \mathbb{R}^{n_j \times r^{\mathrm{t}}_{j}}, j \in \overline{1,d_t}$, 
		and $U\in \mathbb{R}^{n_s \times  r^{\mathrm{t}}_{s}}$ -- the shared factor.
		The tensor diagram for \eqref{eq:sfett_format} is presented in 
		Figure~\ref{fig:sfett_format}.
		\begin{figure}[h!]
			\centering
			\resizebox{0.9\textwidth}{!}{
                \begin{tikzpicture}
                	\begin{pgfonlayer}{nodelayer}
                		\node [style={medium_circ}, minimum size=20mm] (0) at (-7, -0.5) {$W^{(1)}$};
                		\node [style={medium_circ}, minimum size=20mm] (2) at (0, -0.5) {$W^{(d_t)}$};
                		\node [style=none] (3) at (7, -0.5) {$\dotsc$};
                		\node [style={medium_circ}, minimum size=20mm] (5) at (10.5, -0.5) {$W^{(d)}$};
                		\node [style=none] (7) at (7.5, -0.5) {};
                		\node [style=none] (13) at (-7.5, 1.5) {$r^{t}_1$};
                		\node [style=none] (15) at (-0.5, 1.5) {$r^{t}_{d_t}$};
                		\node [style=none] (18) at (-5.25, -1) {$r^{tt}_1$};
                		\node [style=none] (19) at (-1.75, -1) {$r^{tt}_{d_t -1}$};
                		\node [style=none] (20) at (1.75, -1) {$r^{tt}_{d_t}$};
                		\node [style=none] (21) at (9, -1) {$r^{tt}_{d-1}$};
                		\node [style={medium_circ}, minimum size=20mm] (22) at (-7, 3) {$U^{(1)}$};
                		\node [style=none] (23) at (10, 1.5) {$r^{t}_s$};
                		\node [style={medium_circ}, minimum size=20mm] (25) at (0, 3) {$U^{(d_t)}$};
                		\node [style={shared_medium_circ}, minimum size=20mm] (26) at (10.5, 3) {$U$};
                		\node [style=none] (27) at (-7, 5.25) {};
                		\node [style=none] (30) at (0, 5.25) {};
                		\node [style=none] (31) at (10.5, 5.25) {};
                		\node [style=none] (32) at (-7.5, 4.5) {$n_1$};
                		\node [style=none] (35) at (-0.5, 4.5) {$n_{d_t}$};
                		\node [style=none] (36) at (10, 4.5) {$n_s$};
                		\node [style=none] (37) at (-3.5, -0.5) {$\dotsc$};
                		\node [style=none] (38) at (-4, -0.5) {};
                		\node [style=none] (39) at (-3, -0.5) {};
                		\node [style={medium_circ}, minimum size=20mm] (40) at (3.5, -0.5) {$W^{(d_{t}+1)}$};
                		\node [style=none] (41) at (6.5, -0.5) {};
                		\node [style=none] (42) at (3, 1.5) {$r^{t}_s$};
                		\node [style=none] (43) at (5.25, -1) {$r^{tt}_{d_{t+1}}$};
                		\node [style={shared_medium_circ}, minimum size=20mm] (44) at (3.5, 3) {$U$};
                		\node [style=none] (45) at (3.5, 5.25) {};
                		\node [style=none] (46) at (3, 4.5) {$n_{s}$};
                	\end{pgfonlayer}
                	\begin{pgfonlayer}{edgelayer}
                		\draw (5) to (7.center);
                		\draw (22) to (27.center);
                		\draw (25) to (30.center);
                		\draw (26) to (31.center);
                		\draw (0) to (38.center);
                		\draw (39.center) to (2);
                		\draw (40) to (41.center);
                		\draw (44) to (45.center);
                		\draw (2) to (40);
                		\draw (40) to (44);
                		\draw (5) to (26);
                		\draw (0) to (22);
                		\draw (2) to (25);
                	\end{pgfonlayer}
                \end{tikzpicture}
            }
			\caption{SF-ETT decomposition as a tensor diagram}
			\label{fig:sfett_format}
		\end{figure}
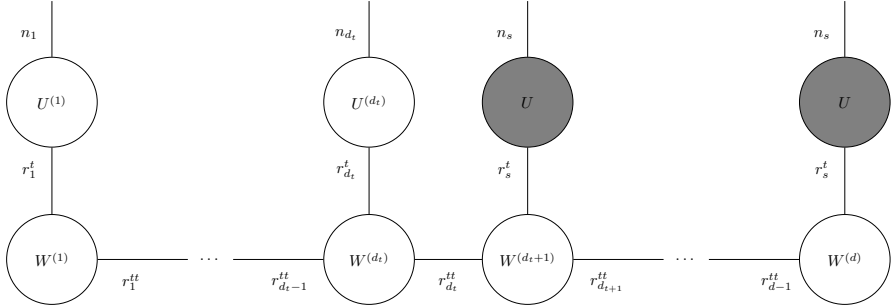

\subsection{SF-ETT ranks}
        Now we need to define the rank for this format.
		By analogy to the other tensor formats, we have that all rank parameters can be minimized independently. Formal verification of this fact is given in the proof the Theorem~\ref{theorem:sfett_ranks}. 
        Thus, the tuple $(r^{\mathrm{tt}}_{1}, \dotsc, r^{\mathrm{tt}}_{d - 1}, r^{\mathrm{t}}_{1}, 
		\dots r^{\mathrm{t}}_{d_t}, r^{\mathrm{t}}_{s})$
		is called the SF-ETT rank if it consists of minimal possible values.
        The exact values in terms of unfolding matrices is given in the theorem below.

		\begin{theorem}\label{theorem:sfett_ranks}
			The SF-ETT rank 
			$(r^{\mathrm{tt}}_{1}, \dotsc, r^{\mathrm{tt}}_{d - 1}, r^{\mathrm{t}}_{1}, 
			\dots r^{\mathrm{t}}_{d_t}, r^{\mathrm{t}}_{s})$ 
			of a tensor $\mathcal{X}$ is given by
			\begin{align}\label{eq:sfett_ranks}
				& r^{\mathrm{tt}}_{i} = \mathrm{rank}\,\left( \mathcal{X}^{<i>} \right), 
				\quad \forall i \in \overline{1,d - 1},\nonumber\\
				& r^{\mathrm{t}}_{i} = \mathrm{rank}\,\left( \mathcal{X}_{(i)} \right),  
				\quad \forall i \in \overline{1,d_t},\nonumber\\
				& r^{\mathrm{t}}_{s} = \mathrm{rank}\,\left(
				\begin{bmatrix}
					\mathcal{X}_{(d_t + 1)} &
					 \mathcal{X}_{(d_t + 2)} & \dots & \mathcal{X}_{(d)}\nonumber
				\end{bmatrix}\right).
			\end{align}
		\end{theorem}
		\begin{proof}
			Indeed, let $\mathcal{X}$ tensor be represented in the SF-ETT format
			\begin{equation*}
				\mathcal{X} = \left\llbracket \mathcal{G}; 
				U^{(1)}, \dotsc, U^{(d_t)}, U, \dotsc, U\right\rrbracket, \quad 
				\mathcal{G} = \mathcal{T}\left( W^{(1)}, \dotsc, W^{(d)} \right).
			\end{equation*}
			Due to the fact that SF-ETT format is a specific case of the Tucker decomposition, the 
			following equation holds
			\begin{equation*}
				\mathcal{X}_{(k)} = U^{(k)} \mathcal{G}_{(k)} \left( U   
					\otimes_K \dotsm \otimes_K U^{(k + 1)} \otimes_K U^{(k - 1)}\otimes_K\dotsm \otimes_K U^{(1)}
				 \right)^\top,
			\end{equation*}
			which is a skeleton decomposition. This means that 
			\begin{equation*}
				\mathrm{rank}\, \left( \mathcal{X}_{(k)} \right) \le r_{k}^t.
			\end{equation*}
			The shared Tucker rank formula arises from the same proposition in \citep{peshekhonov2024training}. 
			As a result, one may take $U^{(i)} \in \mathbb{R}^{n_i \times r^t_i}$ 
			as basis columns in $\mathrm{Im}\,\left( \mathcal{X}_{(i)}\right)$ for 
			every $i \in \overline{1, d_t}$	and $U \in \mathbb{R}^{n \times r^t_s}$ as basis columns
			in 	$\mathrm{Im}\left(
				\begin{bmatrix} \mathcal{X}_{(d_t + 1)} & \mathcal{X}_{(d_t + 2)} & \dots & \mathcal{X}_{(d)}
			\end{bmatrix}\right)$ then the Tucker core can be calculated as follows:
			\begin{equation*}
				\mathcal{G} = \left\llbracket \mathcal{X}; {U^{(1)}}^\top, \dotsc,{U^{(d_t)}}^\top, 
				U^\top, \dotsc, U^\top 
				\right\rrbracket \in \mathbb{R}^{r^t_1 \times 
				\dotsm \times r^t_{d_t} \times r^t_s \times \dotsm \times r^t_s}.
			\end{equation*}
			On the other hand, the formula for the $k$-th unfolding in the ETT format is
			\begin{equation}\label{eq:proof_sfett_ranks_1}
				 \mathcal{X}^{<k>} = 
				\left[\left(U^{(k)} \otimes_K \dotsm \otimes_K U^{(1)}\right)
				\mathcal{G}_{ \le k} \right] \\
				 \left[\mathcal{G}_{ \ge k + 1}^\top
				\left(\left(U^{(d)}\right)^\top \otimes_K  \dotsm \otimes_K \left(U^{(k + 1)}\right)^\top\right)\right],
			\end{equation}
			where $U^{(k)} = U, k > d_t$ and explicit formulas 
			for $\mathcal{G}_{ \le \mu}, \mathcal{G}_{ \ge \mu }$ written 
			in \eqref{eq:unf_less_mu}, \eqref{eq:unf_more_mu}.
			It is a skeleton decomposition for the unfolding. 
            Moreover from \eqref{eq:proof_sfett_ranks_1},
			due to the full column rank of $U^{(i)}$ one can notice that  
			\begin{equation*}
				\mathrm{rank}\, \left( \mathcal{X}^{<k>} \right) = \mathrm{rank}\, \left( \mathcal{G}^{<k>} \right)
				 \le r_{k}^{tt}.
			\end{equation*}
			Decomposing the core $\mathcal{G}$ into the TT decomposition with the TT rank equal to 
			$(\mathcal{X}^{<1>}, \dotsc, \mathcal{X}^{<d-1>})$~\citep{oseledets2010tt}, one may obtain  the
			SF-ETT decomposition with the parameters that achieve the proposed lower bounds.
		\end{proof}

		The SF-ETT decomposition is not unique and can be stored in orthogonalized form.  
		A decomposition of $\mathcal{X}$ that consists 
		of orthogonal Tucker factors $ U, U^{(i)}, i \in \overline{1, d_t}$
		and $\mu$-orthogonal TT decomposition of Tucker core is denoted as  a $\mu$-orthogonal SF-ETT form
		\begin{equation}\label{eq:sfett_format_ortho}
			\begin{aligned}
				&\mathcal{X} = \left\llbracket \mathcal{G}; U^{(1)}, \dotsc U^{(d_t)}, U, \dotsc, U \right\rrbracket,\\
				&\mathcal{G} = \mathcal{T}
				\left( W^{(1)}_L, \dotsc, W^{(\mu - 1)}_L, \overline{W}^{(\mu)}, W^{(\mu + 1)}_R, \dotsc, W^{(d)}_R \right),
			\end{aligned}
		\end{equation}
		where $W_L^{(i)}, i \in \overline{1, \mu - 1}$ are left orthogonal, 
		$W_R^{(i)}, i \in \overline{\mu + 1, d}$ are right orthogonal
		and $\overline{W}^{(\mu)}$ is unconstrained. This can be achieved via applying successive 
		QR decomposition operations to Tucker factors and TT cores.
		The tensor diagram for the $\mu$-orthogonal SF-ETT tensor form is represented in \eqref{fig:sfett_mu_orth_format}.

		\begin{figure}[h!tp]
			\centering
			\resizebox{0.9\textwidth}{!}{
                \begin{tikzpicture}
                	\begin{pgfonlayer}{nodelayer}
                		\node [style={medium_circ}, minimum size=20mm] (0) at (-7, -0.5) {};
                		\fill [black] (-7, -0.5) -- ++(45:10mm) arc (45:-135:10mm) -- cycle;
                		\node [style={medium_circ}, minimum size=20mm] (2) at (0, -0.5) {};
                		\node [style=none] (3) at (10.5, -0.5) {$\dotsc$};
                		\node [style={medium_circ}, minimum size=20mm] (5) at (14, -0.5) {};
                		\fill [black] (14, -0.5) -- ++(135:10mm) arc (135:315:10mm) -- cycle;
                		\node [style={medium_circ}, minimum size=20mm] (2) at (0, -0.5) {};
                		\fill [black] (0, -0.5) -- ++(45:10mm) arc (45:-135:10mm) -- cycle;
                		\node [style=none] (7) at (11, -0.5) {};
                		\node [style=none] (13) at (-7.5, 1.5) {$r^{t}_1$};
                		\node [style=none] (15) at (-0.5, 1.5) {$r^{t}_{\mu - 1}$};
                		\node [style=none] (18) at (-5.5, -1) {$r^{tt}_1$};
                		\node [style=none] (19) at (-1.5, -1) {$r^{tt}_{\mu -2}$};
                		\node [style=none] (20) at (1.5, -1) {$r^{tt}_{\mu - 1}$};
                		\node [style=none] (21) at (12.5, -1) {$r^{tt}_{d-1}$};
                		\node [style={medium_circ}, minimum size=20mm] (22) at (-7, 3) {};
                		\fill [black] (-7, 3) -- ++(180:10mm) arc (180:360:10mm) -- cycle;
                		\node [style=none] (23) at (13.5, 1.5) {$r^{t}_s$};
                		\node [style={medium_circ}, minimum size=20mm] (25) at (0, 3) {};
                		\fill [black] (0, 3) -- ++(180:10mm) arc (180:360:10mm) -- cycle;
                		\node [style={shared_medium_circ}, minimum size=20mm] (26) at (14, 3) {};
                		\fill [black] (14, 3) -- ++(180:10mm) arc (180:360:10mm) -- cycle;
                		\node [style=none] (27) at (-7, 5.25) {};
                		\node [style=none] (30) at (0, 5.25) {};
                		\node [style=none] (31) at (14, 5.25) {};
                		\node [style=none] (32) at (-7.5, 4.5) {$n_1$};
                		\node [style=none] (35) at (-0.5, 4.5) {$n_{\mu -1}$};
                		\node [style=none] (36) at (13.5, 4.5) {$n_s$};
                		\node [style=none] (37) at (-3.5, -0.5) {$\dotsc$};
                		\node [style=none] (38) at (-4, -0.5) {};
                		\node [style=none] (39) at (-3, -0.5) {};
                		\node [style={medium_circ}, minimum size=20mm] (40) at (7, -0.5) {};
                		\fill [black] (7, -0.5) -- ++(135:10mm) arc (135:315:10mm) -- cycle;
                		\node [style=none] (41) at (10, -0.5) {};
                		\node [style=none] (42) at (6.5, 1.5) {$r^{t}_{\mu + 1}$};
                		\node [style=none] (43) at (8.75, -1) {$r^{tt}_{\mu + 1}$};
                		\node [style={medium_circ}, minimum size=20mm] (44) at (7, 3) {};
                		\fill [black] (7, 3) -- ++(180:10mm) arc (180:360:10mm) -- cycle;
                		\node [style=none] (45) at (7, 5.25) {};
                		\node [style=none] (46) at (6.5, 4.5) {$n_{\mu + 1}$};
                		\node [style={medium_circ}, minimum size=20mm] (47) at (3.5, -0.5) {};
                		\node [style={medium_circ}, minimum size=20mm] (49) at (3.5, 3) {};
                		\fill [black] (3.5, 3) -- ++(180:10mm) arc (180:360:10mm) -- cycle;
                		\node [style=none] (50) at (3.5, 5.25) {};
                		\node [style=none] (51) at (5, -1) {$r^{tt}_{\mu}$};
                		\node [style=none] (52) at (2.75, 4.5) {$n_{\mu}$};
                		\node [style=none] (53) at (2.75, 1.5) {$r^{t}_{\mu}$};
                		\node [style=none] (54) at (-7, -2) {$1$};
                		\node [style=none] (55) at (0, -2) {$\mu - 1$};
                		\node [style=none] (56) at (3.5, -2) {$\mu$};
                		\node [style=none] (57) at (7, -2) {$\mu + 1$};
                		\node [style=none] (58) at (14, -2) {$d$};
                	\end{pgfonlayer}
                	\begin{pgfonlayer}{edgelayer}
                		\draw (5) to (7.center);
                		\draw (22) to (27.center);
                		\draw (25) to (30.center);
                		\draw (26) to (31.center);
                		\draw (0) to (38.center);
                		\draw (39.center) to (2);
                		\draw (40) to (41.center);
                		\draw (44) to (45.center);
                		\draw (40) to (44);
                		\draw (5) to (26);
                		\draw (0) to (22);
                		\draw (2) to (25);
                		\draw (49) to (50.center);
                		\draw (47) to (49);
                		\draw (2) to (47);
                		\draw (47) to (40);
                	\end{pgfonlayer}
                \end{tikzpicture}
            }		
			\caption{The visualization of a SF-ETT tensor in $\mu$-orthogonal form}
			\label{fig:sfett_mu_orth_format}
		\end{figure}
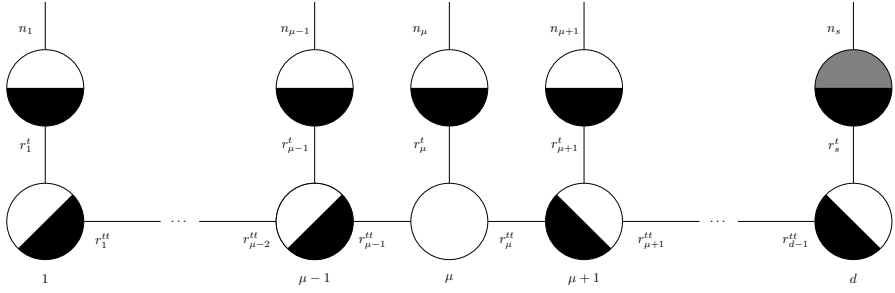

		Here and further all SF-ETT decomposed tensors should be 
		considered as $(d - 1)$-orthogonal unless otherwise stated.
		The amount of parameters to store SF-ETT decomposed tensor $\mathcal{X}$ is 
		\begin{equation}\label{eq:sfett_param_amount}
			\mathcal{O}\left(d{(r^{\mathrm{tt}})}^2 r^{\mathrm{t}}
			+ d_t r^{\mathrm{t}}n + r^{\mathrm{t}}_{s}n_{s}\right), \quad r^{\mathrm{t}} = \max_i r^{\mathrm{t}}_{i},\quad r^{\mathrm{tt}} = \max_i r^{\mathrm{tt}}_{i},\quad n = \max_{i \le d_t} n_i.
		\end{equation}
        Suppose that $d_s=d$, i.e., all the modes are shared.
        For simplicity let us also assume that all the ranks are equal to~$r$. 
        We are interested in the regime when the mode size $n_s$ is large, meaning that $n_s \gg r^2$.
        From~\eqref{eq:sfett_param_amount}, we observe that the benefit in storage for the proposed decomposition is proportional to the dimensionality $d$.
		
	\subsection{SF-ETT-SVD and rounding}\label{sub:sfett_svd_rounding}
    Next, we need a reliable algorithm to approximate a tensor in the new format.
    Fortunately, such an algorithm exists and boils down to a number of SVD decompositions. 
    We refer to this algorithm as the SF-ETT-SVD. 
    This algorithm is also quasi-optimal as follows from Theorem~\ref{theorem:sfett_hosvd}, meaning that it can be further used as a retraction in the Riemannian optimization framework.

		\begin{theorem}\label{theorem:sfett_hosvd}
			Let $\mathcal{X} \in \mathbb{R}^{n_1\times \dotsc\times n_{d_t}\times n_{s}\times \dotsc\times n_{s}}$, 
			let also $U^{(i)} \in \mathbb{R}^{n_i \times r^t_i}$ be a matrix of the first $r^t_i$ left singular vectors of 
			$\mathcal{X}_{(i)}$ and $U \in \mathbb{R}^{n_s \times r^t_s}$ of the first $r^t_s$ left singular vectors of 
			$\begin{bmatrix} \mathcal{X}_{(d_t + 1)} &\mathcal{X}_{(d_t + 2)} & \dots & \mathcal{X}_{(d)}\nonumber
			\end{bmatrix}$. Let $\mathcal{G} \in \mathbb{R}^{r^t_1 \times \dotsm r^t_s}$ such that
			\begin{equation*}
				\mathcal{G} = \left\llbracket \mathcal{X}; {U^{(1)}}^{\top}, \dotsc, {U^{(d_t)}}^{\top}, 
				U^\top, \dotsc, U^\top\right\rrbracket.
			\end{equation*}
            Let also $\mathcal{T}\left( W^{(1)}, \dotsc, W^{(d)} \right)$ be the TT-SVD~\citep{oseledets2011tensor} approximation of $\mathcal{G}$ with the TT-rank $( r^{tt}_1, \dots,  r^{tt}_{d-1})$.
			Then for $\mathbf{r} = (r^{\mathrm{tt}}_{1}, \dotsc, r^{\mathrm{tt}}_{d - 1}, r^{\mathrm{t}}_{1}, 
			\dots r^{\mathrm{t}}_{d_t}, r^{\mathrm{t}}_{s})$ the tensor
			\begin{equation*}
				P_{\mathbf{r}}^{\mathtt{sf-ett-hosvd}} (\mathcal{X}) 
				= \left\llbracket 
				\mathcal{T}\left( W^{(1)}, \dotsc, {W}^{(d)}\right); 
				U^{(1)}, \dotsc, U^{(d_t)}, U, \dotsc, U
				\right\rrbracket
			\end{equation*}
			is a quasi-optimal approximation of $\mathcal{X}$:
			\begin{equation*}
				\left\|\mathcal{X} - P_{\mathbf{r}}^{\mathtt{sf-ett-hosvd}} (\mathcal{X})\right\| \le 
				C(d) 
				\inf\limits_{\mathrm{rank}_{\mathrm{sfett}}^{d_s}(\mathcal{Y}) 
				\preceq \mathbf{r}} \| \mathcal{X} - \mathcal{Y}\|,
			\end{equation*}
			where $\preceq$ is the element-wise comparison and 
			$
				C(d) = \left( \sqrt{d} +  \sqrt{d}\sqrt{d - 1} + \sqrt{d - 1} \right).
			$
		\end{theorem}
        
		\begin{proof}
			Let $\mathcal{X} \in \mathbb{R}^{n_1\times \dotsc\times n_{d_t}, n_{s}, \dotsc, n_{s}}$ be an arbitrary tensor 
			and $P_{\mathbf{r}}^{\mathtt{sf-ett-hosvd}} (\mathcal{X})$  denoted as $P(\mathcal{X})$
			its SF-ETT approximation of rank $\mathbf{r} = (\mathbf{r}^{\mathrm{tt}}, \mathbf{r}^{\mathrm{t}})$ 
			implemented via applying SF-HOSVD to TT compressed tensor denoted as $P_\mathrm{TT}(\mathcal{X})$, then
			\begin{align*}
				&	\|\mathcal{X} - P(\mathcal{X})\| \le \|\mathcal{X} - P_{\mathrm{TT}}(\mathcal{X}) \| + 
				\|P_{\mathrm{TT}}(\mathcal{X}) - P(\mathcal{X}) \| \le  \\
				&\sqrt{d - 1} \inf_{\mathcal{Y}\colon  \mathrm{TTrank}\, (\mathcal{Y}) \preceq \mathbf{r}^{\mathrm{tt}}} 
				\| \mathcal{X} - \mathcal{Y}\| + \|P_{\mathrm{TT}}(\mathcal{X}) - P(\mathcal{X}) \| \le \\
				&\sqrt{d - 1} \inf_{\mathcal{Y}\colon  \mathrm{SFETTrank}\, (\mathcal{Y}) \preceq \mathbf{r}} 
				\| \mathcal{X} - \mathcal{Y}\|	+ \|P_{\mathrm{TT}}(\mathcal{X}) - P(\mathcal{X}) \|.
			\end{align*}
			\begin{align*}
				& \|P_{\mathrm{TT}}(\mathcal{X}) - P(\mathcal{X}) \| \le  
				\sqrt{d} \inf_{\mathcal{Y}\colon  \mathrm{SFrank}\, (\mathcal{Y}) \preceq \mathbf{r}^{\mathrm{t}}} 
				\| P_{\mathrm{TT}}(\mathcal{X}) - \mathcal{Y}\| =  \\
				&\sqrt{d} \inf_{\mathcal{Y}\colon  \mathrm{SFrank}\, (\mathcal{Y}) \preceq \mathbf{r}^{\mathrm{t}}} 
				\| P_{\mathrm{TT}}(\mathcal{X}) - \mathcal{X} + \mathcal{X} - \mathcal{Y}\| \le \\
				&\sqrt{d}\|P_{\mathrm{TT}}(\mathcal{X}) - \mathcal{X} \| + 
				\sqrt{d}\inf_{\mathcal{Y}\colon  \mathrm{SFrank}\, (\mathcal{Y}) \preceq \mathbf{r}^{\mathrm{t}}} 
				\| \mathcal{X} - \mathcal{Y}\| \le \\
				&\sqrt{d}\sqrt{d - 1} \inf_{\mathcal{Y}\colon  \mathrm{SFETTrank}\, (\mathcal{Y}) \preceq \mathbf{r}} 
				\| \mathcal{X} - \mathcal{Y}\| + 
				\sqrt{d}\inf_{\mathcal{Y}\colon  \mathrm{SFETTrank}\, (\mathcal{Y}) \preceq \mathbf{r}} 
				\| \mathcal{X} - \mathcal{Y}\|.
			\end{align*}

			Therefore
			\begin{equation*}
				\|\mathcal{X} - P(\mathcal{X})\| \le \left(\sqrt{d} + \sqrt{d}\sqrt{d - 1 } +\sqrt{d -1} \right)
				\inf_{\mathcal{Y}\colon  \mathrm{SFETTrank}\, (\mathcal{Y}) \preceq \mathbf{r}} 
				\| \mathcal{X} - \mathcal{Y}\|.
			\end{equation*}

			Now let $P(\mathcal{X}) = P_{\mathbf{r}}^{\mathtt{sf-ett-hosvd}} (\mathcal{X})$ 
			be a SF-ETT approximation of rank $\mathbf{r} = (\mathbf{r}^{\mathrm{tt}}, \mathbf{r}^{\mathrm{t}})$ 
			obtained by applying TT compression onto core 
			$\mathcal{G}$ of SF-Tucker approximation $P_{\mathrm{SF}}(\mathcal{X})$ 
			of $\mathcal{X}$  with rank equal to $\mathbf{r}^{\mathrm{t}}$, then
			\begin{align*}
				&\|\mathcal{X} - P(\mathcal{X})\| \le \|\mathcal{X} - P_{SF}(\mathcal{X}) \| + 
				\|P_{\mathrm{SF}}(\mathcal{X}) - P(\mathcal{X}) \| \le  \\
				&\sqrt{d} \inf_{\mathcal{Y}\colon  \mathrm{SFrank}\, (\mathcal{Y}) \preceq \mathbf{r}^{\mathrm{t}}} 
				\| \mathcal{X} - \mathcal{Y}\| + \|P_{\mathrm{SF}}(\mathcal{X}) - P(\mathcal{X}) \| \le \\
				&\sqrt{d} \inf_{\mathcal{Y}\colon  \mathrm{SFETTrank}\, (\mathcal{Y}) \preceq \mathbf{r}} 
				\| \mathcal{X} - \mathcal{Y}\|	+ \|P_{\mathrm{SF}}(\mathcal{X}) - P(\mathcal{X}) \|.
			\end{align*}

			Now if $P_{\mathrm{SF}} \left( \mathcal{X} \right) = \llbracket \mathcal{G}; U^{(1)}, \dotsc, U  \rrbracket$, 
			where $U^{(1)}, \dotsc, U$ --- orthogonal matrices then
			\begin{align*}
				& \|P_{\mathrm{SF}}(\mathcal{X}) - P(\mathcal{X}) \| =  
				\left\| \left\llbracket  \mathcal{G}; U^{(1)}, \dotsc, U\right\rrbracket 
				- \left\llbracket P_{\mathrm{TT}}\left(\mathcal{G}\right); U^{(1)}, \dotsc, U\right\rrbracket\right\| = \\
				&\left\| \left\llbracket  \mathcal{G}; U^{(1)}, \dotsc, U\right\rrbracket 
				-P_{\mathrm{TT}}\left(\left\llbracket \mathcal{G}; U^{(1)}, \dotsc, U\right\rrbracket\right)\right\| 
				\le \\
				&\sqrt{d - 1} \inf_{\mathcal{Y}\colon  \mathrm{TTrank}\, (\mathcal{Y}) \preceq \mathbf{r}^{\mathrm{tt}}} 
				\| P_{\mathrm{SF}}\left(\mathcal{X} \right) - \mathcal{Y}\| \le  \\
				&\sqrt{d - 1} \inf_{\mathcal{Y}\colon  \mathrm{SFETTrank}\, (\mathcal{Y}) 
				\preceq \mathbf{r}} 
				\left\| P_{\mathrm{SF}}\left(\mathcal{X} \right) 
				- \mathcal{Y}\right\|\le \sqrt{d - 1}\|P_{\mathrm{SF}}(\mathcal{X}) - \mathcal{X}\| +  \\
				&\sqrt{d - 1} \inf_{\mathcal{Y}\colon  \mathrm{SFETTrank}\, (\mathcal{Y}) 
				\preceq \mathbf{r}} 
				\left\| \mathcal{X} - \mathcal{Y}\right\| \le \\
				&(\sqrt{d} \sqrt{d - 1} + \sqrt{d - 1})\inf_{\mathcal{Y}\colon  
				\mathrm{SFETTrank}\, (\mathcal{Y}) \preceq \mathbf{r}} 
				\| \mathcal{X} - \mathcal{Y}\|.
			\end{align*}

			Therefore
			\begin{equation*}
				\|\mathcal{X} - P(\mathcal{X})\| \le \left(\sqrt{d} + \sqrt{d}\sqrt{d - 1 } +\sqrt{d -1} \right)
				\inf_{\mathcal{Y}\colon  \mathrm{SFETTrank}\, (\mathcal{Y}) \preceq \mathbf{r}} 
				\| \mathcal{X} - \mathcal{Y}\|.
			\end{equation*}

		\end{proof}

        	\begin{remark}\label{remark:sfett_hosvd_two_ways}
			The compression algorithm justified by Theorem~\ref{theorem:sfett_hosvd} can be implemented in two ways. The
			first way is to compress tensor $\mathcal{X}$ into TT decomposition
			\begin{equation*}
				\hat{\mathcal{X}} = \mathcal{T}\left(\hat{W}^{(1)}, \dotsc,\hat{W}^{(d)}\right)
			\end{equation*}
			and then apply SF-HOSVD algorithm to the tensor
			\begin{equation*}
				\hat{\mathcal{X}} = 
				\left\llbracket\mathcal{T}\left(\hat{W}^{(1)}, \dotsc,\hat{W}^{(d)}\right); I, \dotsc, I\right\rrbracket.
			\end{equation*}
			The second way is to compress the input tensor via SF-HOVSD first and then apply TT 
			compression to the core of the obtained approximation. 
			Both approaches will ensure quasi-optimal approximation with the same quasi-optimality constant.
		\end{remark}

		If one has a tensor already in the SF-ETT format, it is possible to efficiently apply SF-ETT-SVD without forming full tensors. This operation is called SF-ETT-Rounding and has asymptotic complexity
		\begin{equation*}
			\mathcal{O}\left(d (r^{\mathrm{tt}})^2 r^{\mathrm{t}}(r^{\mathrm{tt}}  + r^{\mathrm{t}}) + 
			N (r^{\mathrm{t}})^2 (d_{\mathrm{t}} + 1)\right),
		\end{equation*}
		where 
		$r^{\mathrm{t}} = \max_i r^{\mathrm{t}}_{i}$, $r^{\mathrm{tt}} = \max_j r^{\mathrm{tt}}_{j}$, $N = \max_i n_i$.

		The SF-ETT Rounding algorithm is described in Algorithm~\ref{alg:sfett_rounding} and can be logically
		split into two stages: TT and SF-Tucker roundings. The first stage is fully described in the first line:
		TT rank is rounded to the target rank via $\mathtt{TTRound}$ from \citep{oseledets2010tt}. 
		All of the other steps are devoted to the second stage.

		This stage begins with the computation of unconstrained core $\overline{W}^{(k)}$ 
		in $k$-or-thogonal representation of 
		SF-ETT tensor for every $k \in \overline{1, d}$ defined in \eqref{eq:sfett_format_ortho} (rows $2-6$).
		According to \citep{peshekhonov2024training}, the algorithm should compute SVD of 
		$\mathcal{X}_{(k)}$ if $k \le d_t$ and SVD of $\begin{bmatrix} 
			\mathcal{X}_{(d_t + 1)} &\mathcal{X}_{(d_t + 2)} & \dots & \mathcal{X}_{(d)}\nonumber
			\end{bmatrix}$. 
		The main idea is to obtain this decomposition without forming dense Tucker core 
		out of its TT format. Suppose that the tensor $\mathcal{X}$ is stored in the $k$-orthogonal format 
		defined in \eqref{eq:sfett_format_ortho} 
		\begin{equation*}
			\mathcal{X}	= \left\llbracket 
			\mathcal{T}\left( W_L^{(1)}, \dotsc, \overline{W}_L^{(k)}, \dotsc, W_R^{(k)} \right);
			U^{(1)}, \dotsc, U
			 \right\rrbracket.
		\end{equation*}
		One may write its $k$-th matricization for $k \le d_t$ as follows
		\begin{equation*}
			\mathcal{X}_{(k)} = U^{(k)} \overline{W}^{(k)}_{(2)} M^{(k)},
		\end{equation*}
		where
		\begin{multline*}
			M^{(k)} = \\
			\mathcal{G}_{\ge k+1}^\top \left( U^\top \otimes_K \dotsm \otimes_K \left(U^{(k + 1)}\right)^\top \right) 
				\otimes_K 
				\mathcal{G}_{ \le k - 1}^\top\left( \left(U^{(k - 1)}\right)^\top 
			\otimes_K \dotsm \otimes_K \left(U^{(1)}\right)^\top \right)
		\end{multline*}
		is a matrix with orthonormal rows, so 
		\begin{equation*}
			\mathtt{SVD}\left(\mathcal{X}_{(k)} \right)=U^{(k)} \mathtt{SVD}\left(\overline{W}^{(k)}_{(2)} \right) M_k.
		\end{equation*}
		This structure is also preserved for the shared Tucker rank:
		\begin{align*}
			&\begin{bmatrix} 
				\mathcal{X}_{(d_t + 1)} & \dots & \mathcal{X}_{(d)}
			\end{bmatrix} = U\begin{bmatrix} 
			 \overline{W}^{(d_t + 1)}_{(2)}   &
			\dots & 
			 \overline{W}^{(d)}_{(2)}   
			\end{bmatrix}
            \mathtt{bdiag}(M^{(d_t + 1)}, \dots, M^{(d)}).
		\end{align*}
		Due to orthogonality of $U$ and a block-diagonal matrix on the right hand side, the SVD of $\begin{bmatrix} \mathcal{X}_{(d_t + 1)} &\mathcal{X}_{(d_t + 2)} & \dots & \mathcal{X}_{(d)}
			\end{bmatrix}$ can be obtained from the SVD of $\begin{bmatrix} 
			 \overline{W}^{(d_t + 1)}_{(2)}   &
			 \overline{W}^{(d_t + 2)}_{(2)}   &
			\dots & 
			 \overline{W}^{(d)}_{(2)}   
			\end{bmatrix}$.
		The lines from $7$ to $12$ are devoted to the rounding of SF-Tucker modes corresponding 
		to the non-shared Tucker factors via the SVD truncation. 
        The rounding for the shared SF-Tucker mode is described in the lines $13$ to~$18$.

		\begin{algorithm}[tp]
			\caption{SF-ETT Rounding}
			\label{alg:sfett_rounding}
			\textbf{Require:}
			\begin{enumerate}
				\item[]$\mathcal{X} = \left\llbracket \mathcal{T}\left( W^{(1)}, \dotsc, W^{(d)}\right); 
					U^{(1)}, \dotsc, U^{(d_t)}, U, \dotsc, U\right\rrbracket$ --- $d$-orthogonal SF-ETT decomposition.
				\item[]$\mathbf{r}=
					\left(r^{\mathtt{tt}}_{1}, \dotsc, r^{\mathtt{tt}}_{d}, r^{\mathtt{t}}_{1}, 
					\dots r^{\mathtt{t}}_{d_t}, r^{\mathtt{t}}_{s}\right)
					= \left(\mathbf{r}^{\mathtt{tt}}, \mathbf{r}^{\mathtt{t}}\right)$ --- target rank.
			\end{enumerate}
			\textbf{Ensure:}
			\begin{enumerate}
				\item[]  $P_{\mathbf{r}}^{\mathtt{sf-ett-hosvd}} ( \mathcal{X} )$ --- SF-ETT with target rank.
			\end{enumerate}
			\textbf{Function:}
			\begin{algorithmic}[1]
				\State $\left[\hat{W}^{(1)},  \dotsc, \hat{W}^{(d)}\right] 
				:= \mathtt{TTRound}\left(\left[W^{(1)}_L, \dotsc, W^{(d)}\right], 
				\mathbf{r}^{\mathtt{tt}}\right)$;
				\Comment{$\mathcal{O}\left(d r^{\mathtt{t}}(r^{\mathtt{tt}})^3\right)$}
				\State $\mathtt{NonOrthCores} := \left[\hat{W}^{(1)}\right]$;
				\For {$i = 2, \dots, d$}
					\State $\left[\hat{W}^{(1)}, \dotsc, \hat{W}^{(d)}\right] := 
					\mathtt{OrthogonalizeTT}\left(\left[\hat{W}^{(1)}, \dotsc, \hat{W}^{(d)}\right], i\right)$;
					\Comment{$\mathcal{O}\left(r^{\mathtt{t}}\left(r^{\mathtt{tt}}\right)^3\right)$}
					\State $\mathtt{NonOrthCores} := [\mathtt{NonOrthCores}, \hat{W}^{(i)}]$;
				\EndFor 
				\For {$i = 1, \dots, d_t$}
					\State $\mathtt{MatricizedCore} := \mathtt{Matricize}\left(\mathtt{NonOrthCores}[i], 2\right)$;
					\State $Y, \Sigma, V^\top := \mathtt{SVD}\left(\mathtt{MatricizedCore}, r^{\mathtt{t}}_{i}\right)$;
					\Comment{$\mathcal{O}\left(\left(r^{\mathtt{t}}r^{\mathtt{tt}}\right)^2\right)$}
					\State $W^{(i)} := \left\llbracket \hat{W}^{(i)}; I, Y^\top, I \right\rrbracket$; 
					\Comment{$\mathcal{O}\left((r^\mathtt{t}r^\mathtt{tt})^2\right)$}
					\State $U^{(i)} := U^{(i)}Y$; 
					\Comment{$\mathcal{O}\left(\mathtt{N}(r^\mathtt{t})^2\right)$}
				\EndFor
				\State $\mathtt{SharedFactor} =
				\left[ \mathtt{Matricize}\left(\mathtt{NonOrthCores}[d_{t} + 1], 2\right) \right., \dotsc,$ \\
					\hspace{2.7cm}$\left.\mathtt{Matricize}(\mathtt{NonOrthCores}[d],2) \right];$
				\State $Y, \Sigma, V^\top := \mathtt{SVD}(\mathtt{SharedFactor}, r^{\mathtt{t}}_{s})$;
				\Comment{$\mathcal{O}(d_s(r^{\mathtt{t}}r^{\mathtt{tt}})^2)$}
				\State $U:= UY$; \Comment{$\mathcal{O}(N(r^t)^2)$}
				\For {$i = d_{t + 1}, \dots, d$}
					\State $W^{(i)} := \left\llbracket \hat{W}^{(i)}; I, Y^\top, I \right\rrbracket$;
					\Comment{$\mathcal{O}\left((r^{\mathtt{tt}}r^\mathtt{t})^2\right)$} 
				\EndFor
				\State\Return $\left\llbracket \mathcal{T}\left( W^{(1)}, \dotsc, W^{(d)}\right); 
					U^{(1)}, \dotsc, U^{(d_t)}, U, \dotsc, U\right\rrbracket$.
			\end{algorithmic}
		\end{algorithm}

\section{Shared Factor Extended Tensor Train manifold}\label{sec:sfett_manif}
		The set of fixed rank $\mathbf{r}^{t}$ Tucker tensors
			$\mathrm{T}\, (\mathbf{r}^{t}) 
			\subseteq \mathbb{R}^{n_1\times \dotsc\times n_{d_t}\times n_{s}\times \dotsc\times n_s}$
		forms a smooth embedded submanifold \citep{steinlechner2016riemannian}.  
        In what follows, we assume that $\mathbf{r}^{t} = (r^t_1,\dots,r^t_{d_t}, r^t_s,\dots, r^t_s)$ and introduce the notation $\mathbf{r}^{ts} = (r^t_1,\dots,r^t_{d_t}, r^t_s)$ for the SF-Tucker rank. Then the set of fixed rank $\mathbf{r}^{ts}$ SF-Tucker tensors $
			\mathrm{SFT}\, (\mathbf{r}^{ts})
			\subset \mathrm{T}\, (\mathbf{r}^{t})$
		is also a smooth manifold \citep{peshekhonov2024training}
		embedded into the linear space 
		$\mathbb{R}^{n_1\times \dotsc\times n_{d_t}\times n_{s}\times \dotsc\times n_s}$. Analogically
			$\mathrm{ETT}\, (\mathbf{r}^{tt}, \mathbf{r}^{t})
			\subset \mathrm{T}\, (\mathbf{r}^{t})$
		denotes a smooth tensor manifold of fixed 
		ETT rank $(\mathbf{r}^{tt}, \mathbf{r}^{t})$, which is the HT manifold~\citep{uschmajew2013geometry} 
		embedded into the linear space 
		$\mathbb{R}^{n_1\times \dotsc\times n_{d_t}\times n_{s}\times \dotsc\times n_s}$.
		Finally,
			$\mathrm{SFETT}\, (\mathbf{r}^{tt}, \mathbf{r}^{ts}) 
			\subset \mathrm{T}\, (\mathbf{r}^{t})$
		stands for the set of tensors of fixed SF-ETT rank 
		$(\mathbf{r}^{tt}, \mathbf{r}^{ts})$.
		\subsection{Smoothness of the SF-ETT manifold}%
		\label{sub:sfett_manifold_smoothness}
        The purpose of this section is to introduce a smooth structure on the set of all tensors that can be represented in the SF-ETT 
		format with the fixed SF-ETT rank.
		
		\begin{lemma}\label{lemma:sfett_as_intersection}
			Under the rank constraints above, the $\mathrm{SFETT}$ set represents the following intersection:
			\begin{equation*}
				\mathrm{SFETT}\, (\mathbf{r}^{tt}, \mathbf{r}^{ts}) = \mathrm{SFT}\,(\mathbf{r}^{ts})
				\cap\mathrm{ETT}\, (\mathbf{r}^{tt}, \mathbf{r}^{t}).
			\end{equation*}
		\end{lemma}
		\begin{proof}
			Let $\mathcal{X} \in \mathrm{SFETT}\, (\mathbf{r}^{tt}, \mathbf{r}^{ts})$. 
            Then by the definition of
			the ETT format:
			\begin{equation*}
				\mathcal{X}	= 
				\left\llbracket 
					\mathcal{T} \left(W^{(1)}, \dotsc,W^{(d)}\right);
					U^{(1)}, \dotsc,  U^{(d_t)}, U, \dotsc, U
				\right\rrbracket 
				\Longrightarrow \mathcal{X} \in \mathrm{ETT}\, (\mathbf{r}^{tt}, \mathbf{r}^{t}).
			\end{equation*}
			One can also denote
			\begin{equation*}
				\mathcal{G} = \mathcal{T} \left(W^{(1)}, \dotsc,W^{(d)}\right),
			\end{equation*}
			so
			\begin{equation*}
				\mathcal{X}	= 
				\left\llbracket 
					\mathcal{G};
					U^{(1)}, \dotsc,  U^{(d_t)}, U, \dotsc, U
				\right\rrbracket \in \mathrm{SFT}\,(\mathbf{r}^{ts}).
			\end{equation*}
			This means that:
			\begin{equation*}
				\mathcal{X} \in \mathrm{SFT}\,(\mathbf{r}^{ts})
				\cap \mathrm{ETT}\, (\mathbf{r}^{tt}, \mathbf{r}^{t}).
			\end{equation*}
			Now let $\mathcal{X} \in \mathrm{SFT}\,(\mathbf{r}^{ts})$ and 
			$\mathcal{X} \in \mathrm{ETT}\, (\mathbf{r}^{tt}, \mathbf{r}^{t})$:
			\begin{align*}
					& \mathcal{X}	= 
					\left\llbracket 
						\mathcal{G};
						U^{(1)}, \dotsc,  U^{(d_t)}, U, \dotsc, U
					\right\rrbracket = \\
					& \left\llbracket 
						\mathcal{X};
						U^{(1)}{U^{(1)}}^\top, \dotsc,  U^{(d_t)}{U^{(d_t)}}^\top, UU^\top, \dotsc, UU^\top
					\right\rrbracket,
			\end{align*}
			where $U^{(1)}, \dotsc, U$ have orthonormal columns. Then by the definition of ETT
			\begin{align*}
				& \mathcal{X}	= 
				\left\llbracket 
					\mathcal{X};
					U^{(1)}{U^{(1)}}^\top, \dotsc,  U^{(d_t)}{U^{(d_t)}}^\top, UU^\top, \dotsc, UU^\top
				\right\rrbracket = \\
				& \left\llbracket 
					\mathcal{T}(W^{(1)}, \dotsc,W^{(d)});
					U^{(1)}{U^{(1)}}^\top, \dotsc,  U^{(d_t)}{U^{(d_t)}}^\top, UU^\top, \dotsc, UU^\top
				\right\rrbracket.
			\end{align*}
			Contracting the factors ${U^{(i)}}^\top$ and $U^\top$ with $W^{(i)}$, one can get
			\begin{equation*}
				\mathcal{X}	= 
				\left\llbracket 
					\mathcal{T}(\hat{W}^{(1)}, \dotsc,\hat{W}^{(d)});
					U^{(1)}, \dotsc,  U^{(d_t)}, U, \dotsc, U
				\right\rrbracket,
			\end{equation*}
			and, hence, $\mathcal{X} \in \mathrm{SFETT}\, (\mathbf{r}^{tt}, \mathbf{r}^{ts})$.
		\end{proof}
		The intersection of two smooth manifolds is not necessarily a smooth manifold. Fortunately
		there is a special type of manifold intersection that preserves smoothness.
		\begin{definition}[Transversality~\citep{lee2012transversality}]\label{def:transversality}
			Let $\mathcal{M}$ be a smooth manifold and $\mathcal{A}$, $\mathcal{B}$
			are two embedded submanifolds of $\mathcal{M}$. The $\mathcal{A} \cap \mathcal{B}$
			is a transversal intersection if  $\;\forall p \in \mathcal{A} \cap \mathcal{B}$:
			\begin{equation}
				 T_{p}{\mathcal{A}} + T_{p}{\mathcal{B}} = T_{p}{\mathcal{M}}.
			\end{equation}
		\end{definition}

		\begin{lemma}[Transversality of the intersection]\label{lemma:transversality}
			$\mathrm{SFT}\,(\mathbf{r}^{ts})$ and $\mathrm{ETT}\, (\mathbf{r}^{tt}, \mathbf{r}^{t})$ 
			intersect transversally in $\mathrm{T}\,(\mathbf{r}^{t})$, i.e., 
			$\;\forall \mathcal{X}	\in 
				\mathrm{SFT}\, (\mathbf{r}^{ts})\cap\mathrm{ETT}\, (\mathbf{r}^{tt}, \mathbf{r}^{t})$, we have
			\begin{equation}
				T_{\mathcal{X}} \mathrm{SFT}\, 
				(\mathbf{r}^{ts}) + T_{\mathcal{X}}\mathrm{ETT}\, (\mathbf{r}^{tt}, \mathbf{r}^{t})= 
				T_{\mathcal{X}}\mathrm{T}\, (\mathbf{r}^{t}).
			\end{equation}
		\end{lemma}
		\begin{proof}
			Let $X \in \mathrm{SFETT}\, (\mathbf{r}^{tt}, \mathbf{r}^{ts})$, then
			\begin{equation*}
				\mathcal{X}	= 
				\left\llbracket 
					\mathcal{G};
					U^{(1)}, \dotsc,  U^{(d_t)}, U, \dotsc, U
				\right\rrbracket, \quad 
				\mathcal{G} : = \mathcal{T} \left(W^{(1)}, \dotsc,W^{(d)}\right).
			\end{equation*}
            
			According to \citep{uschmajew2013geometry},
			we have
			\begin{align*}
				&	T_{\mathcal{X}} \mathrm{T} = \left\{ \xi \bigg| 
				\xi =  \left\llbracket 
							\dot{\mathcal{G}}; U^{(1)}, \dotsc,  U^{(d_t)}, U, \dotsc, U
						\right\rrbracket\right. + \\
				&\sum_{i = 1}^{d} 
						\left.\left\llbracket 
						\mathcal{G}; U^{(1)}, \dotsc,  \dot{Z}^{(i)}, \dotsc, U, \dotsc, U
					\right\rrbracket,\quad
				\left(\dot{Z}^{(i)}\right)^\top U^{(i)} = 0, \forall i \in \overline{1,d_t},\right. \\ 
				&\left.\left(\dot{Z}^{(j)}\right)^\top U = 0, 
						\forall j \in \overline{d_t + 1,d}\right\}, 
			\end{align*}
			i.e., every $\xi \in T_{\mathcal{X}} \mathrm{T}$ can be described via the tuple 
				$\left( \dot{\mathcal{G}}, \dot{Z}^{(1)}, \dotsc, \dot{Z}^{(d)}\right)$.
			The tangent space for $\mathrm{SFT}\, (\mathbf{r}^{ts})$ from
			\citep{peshekhonov2024training} can be written as
			\begin{align*}
				&	T_{\mathcal{X}} \mathrm{SFT} = \left\{ \eta \bigg| 
				\eta =  \left\llbracket 
							\dot{\mathcal{G}}; U^{(1)}, \dotsc,  U^{(d_t)}, U, \dotsc, U
						\right\rrbracket\right. + \\
				&\sum_{i = 1}^{d_t} 
						\left.\left\llbracket 
						\mathcal{G}; U^{(1)}, \dotsc,  \dot{U}^{(i)}, \dotsc, U^{(d_t)}, 
						\dotsc, U, \dotsc, U
					\right\rrbracket+\right.\\
				&\sum_{i = d_t + 1}^{d} 
						\left.\left\llbracket 
						\mathcal{G}; U^{(1)}, \dotsc,  \dot{U}^{(d_t)}, U, \dotsc, 
						\underbrace{\dot{U}}_{i},  \dotsc U
					\right\rrbracket,\right.\\
				&\left.\left(\dot{U}^{(i)}\right)^\top U^{(i)} = 0, \forall i \in \overline{1,d_t},\quad
				\dot{U}^\top U = 0\right\},
			\end{align*}
			and, hence, every $\eta \in T_{\mathcal{X}} \mathrm{SFT}$ can be described via the tuple
				$( \dot{\mathcal{G}}, \dot{U}^{(1)}, \dotsc, \dot{U}^{(d_t)}, \dot{U})$.
			The tangent space for $\mathrm{ETT}(\mathbf{r}^{tt}, \mathbf{r}^{t})$ is a particular 
			case of the tangent space for manifold of fixed HT rank described in \citep{uschmajew2013geometry}
			\begin{align*}
				&	T_{\mathcal{X}} \mathrm{ETT} = \left\{ \zeta \bigg| 
				\zeta =  \left\llbracket 
				\dot{\mathcal{T}}_{\mathcal{X}}\left(\dot{W}^{(1)}, \dotsc, \dot{W}^{(d)} \right); 
				U^{(1)}, \dotsc,  U^{(d_t)}, U, \dotsc, U
						\right\rrbracket\right. + \\
				&\sum_{i = 1}^{d} 
						\left.\left\llbracket 
						\mathcal{G}; U^{(1)}, \dotsc,  \dot{Y}^{(i)}, \dotsc, U, \dotsc, U
					\right\rrbracket, \quad
				\left(\dot{Y}^{(i)}\right)^\top U^{(i)} = 0, \forall i \in \overline{1,d_t},\right.\\
					&\left.\left(\dot{Y}^{(j)}\right)^\top U = 0, 
				\forall j \in \overline{d_t + 1,d}, \quad
				L \left(W_L^{(k)} \right)^\top L \left(\dot{W}^{(k)} \right) = 0, k \in 
				\overline{1,d - 1}\right\}, 
			\end{align*}
			so every $\zeta \in T_{\mathcal{X}} \mathrm{ETT}$ can be described via the tuple
				$\left( \dot{W}^{(1)}, \dotsc, \dot{W}^{(d)}, \dot{Y}^{(1)}, \dotsc, \dot{Y}^{(d)}\right)$.

			Now let $\eta \in T_{\mathcal{X}}\mathrm{SFT}$, $\zeta \in T_{\mathcal{X}} \mathrm{ETT}$, then 
			\begin{align*}
				& \eta + \zeta  = \left\llbracket  
			\dot{\mathcal{T}}_{\mathcal{X}}\left(\dot{W}^{(1)}, \dotsc, \dot{W}^{(d)} \right) + \dot{\mathcal{G}}; 
				U^{(1)}, \dotsc,  U^{(d_t)}, U, \dotsc, U
				\right\rrbracket + \\
				& \sum_{i = 1}^{d_t} \left\llbracket  
				\mathcal{G}; 
				U^{(1)}, \dotsc,  \dot{U}^{(i)} +\dot{Y}^{(i)} , \dotsc,  U^{(d_t)}, U, \dotsc, U
				\right\rrbracket  + \\
				& \sum_{i=d_t + 1}^{d} \left\llbracket 
				\mathcal{G}; 
				U^{(1)}, \dotsc, U^{(d_t)}, U, \dots,\underbrace{\dot{U} + \dot{Y}^{(i)}}_{i}, \dotsc, U
				\right\rrbracket
			\end{align*}
			and can be interpreted as element of $T_{\mathcal{X}} \mathrm{T}$ in terms of the following tuple
			\begin{equation*}
				\left( \dot{\mathcal{T}}_{\mathcal{X}}\left(\dot{W}^{(1)}, \dotsc, \dot{W}^{(d)} \right) + \dot{\mathcal{G}}, 
				\dot{U}^{(1)} + \dot{Y}^{(1)}, \dotsc, \dot{U} + \dot{Y}^{(d)}\right).
			\end{equation*}
			As a result, we obtain $T_{\mathcal{X}} \mathrm{ETT} + T_{\mathcal{X}} \mathrm{SFT} \subseteq T_{\mathcal{X}} \mathrm{T}$. 
            
            Let
			$\xi$ be an arbitrary element of $T_{\mathcal{X}} \mathrm{T}$ given by the tuple $
				\left( \dot{\mathcal{G}}, \dot{Z}^{(1)}, \dotsc, \dot{Z}^{(d)}\right)$, 
			then 
			$\exists \; \hat{\eta} \in T_{\mathcal{X}} \mathrm{SFT}, \hat{\zeta} \in T_{\mathcal{X}} \mathrm{ETT}: \hat{\eta} + 
			\hat{\zeta} = \xi$. For example,  $\hat{\eta}$ may be given by a tuple
			\begin{equation*}
				\left( \dot{\mathcal{G}}, \dot{Z}^{(1)}, \dotsc, \dot{Z}^{(d_t)}, \underbrace{0, \dotsc, 0}_{d - d_t}\right)
			\end{equation*}
			and $\hat{\zeta}$ by
			\begin{equation*}
				\left(\underbrace{0, \dots, 0}_{d}, \underbrace{0, \dots, 0}_{d_t}, 
				\dot{Z}^{(d_t + 1)}, \dotsc,  \dot{Z}^{(d)}\right).
			\end{equation*}
			So $T_{\mathcal{X}} \mathrm{ETT} + T_{\mathcal{X}} \mathrm{SFT} \supseteq T_{\mathcal{X}} \mathrm{T}$.
		\end{proof}
		
		Uniting the results of Lemma~\ref{lemma:sfett_as_intersection} and Lemma~\ref{lemma:transversality}, one can obtain  
		the smooth structure on $\mathrm{SFETT}\, (\mathbf{r}^{tt}, \mathbf{r}^{ts})$.
	\begin{theorem}\label{theorem:sfett_manifold_sm_emb}
		The set $\mathrm{SFETT}\, (\mathbf{r}^{tt}, \mathbf{r}^{ts})$	
		forms a smooth manifold embedded 
		 into $\mathbb{R}^{n_1\times \dotsc\times n_{d_t}\times n_{s}\times \dotsc\times n_s}$ 
		 of the dimension
		 \begin{equation*}
			 \dim \left(\mathrm{SFETT}\, (\mathbf{r}^{\mathrm{tt}}, \mathbf{r}^{ts})\right) 
			 = \dim \left(\mathrm{TT}(\mathbf{r}^{tt})\right) +
             \sum\limits_{i = 1}^{d_t} r^t_i (n_i - r^t_i) + r^t_s (n_s - r^t_s),
		 \end{equation*}
		 where $d_s = d - d_t > 0$, $r_0^{tt} = r_d^{tt} = 1$ and
		 \begin{equation*}
			\dim \left(\mathrm{TT}(\mathbf{r}^{tt})\right)	=
			\sum\limits_{i = 1}^{d_t}r_{i - 1}^{tt}r^t_i r^{tt}_i + 
            r^t_s \sum\limits_{i = d_t + 1}^{d}r_{i-1}^{tt}r_{i}^{tt}
            - \sum\limits_{i = 1}^d {(r^{tt}_i)}^2.
		 \end{equation*}
	\end{theorem}
   		\begin{proof}
			First of all, according to
			\citep[Definition $1.37$]{robbin2011introduction}, we have that $\mathrm{ETT}\, (\mathbf{r}^{tt}, \mathbf{r}^{t})$ and
			$\mathrm{SFT}\, ( \mathbf{r}^{ts})$ are smooth manifolds embedded into the $\mathrm{T}\, (\mathbf{r}^{t})$, which in turn is embedded into 
			$\mathbb{R}^{n_1\times \dotsc\times n_{d_t}\times n_{s}\times \dotsc\times n_s}$. 		
			Then we use the fact that $\mathrm{SFETT}\, ( \mathbf{r}^{tt} , \mathbf{r}^{ts})$ is the transversal
			intersection of these manifolds in $\mathrm{T}(\mathbf{r}^t)$. 
			Hence, according to \citep{pollack1974differential} $\mathrm{SFETT}\, ( \mathbf{r}^{tt} , \mathbf{r}^{ts})$ is 
			a smooth manifold embedded in $\mathrm{T}(\mathbf{r}^t)$ with the dimension 
			\begin{equation*}
				\dim \left(\mathrm{SFETT}(\mathbf{r}^{tt}, \mathbf{r}^{ts})\right) =  
				\dim \left(\mathrm{ETT}\, (\mathbf{r}^{tt}, \mathbf{r}^{t})\right) +
				\dim \left(\mathrm{SFT}\, (\mathbf{r}^{ts})\right)- 
				\dim \left(\mathrm{T}\, (\mathbf{r}^{t})\right),
			\end{equation*}
			where 
			\begin{align*}
				& \dim \left(\mathrm{ETT}\, (\mathbf{r}^{tt}, \mathbf{r}^{ts})\right) = 
                \dim \left(\mathrm{TT}(\mathbf{r}^{tt})\right)	+ 
                \sum\limits_{i = 1}^{d_t}r_i^t (n_i - r_i^t) + 
                d_s r_s^t(n_s - r_s^t),
				\\
				& \dim \left(\mathrm{SFT}\, (\mathbf{r}^{ts})\right)  = 
                (r_s^t)^{d_s} \prod_{i=1}^{d_t}r_i^t + 
                \sum\limits_{i = 1}^{d_t} r_i^t(n_i - r_i^t) + r_s^t (n_s - r_s^t),\\
				& \dim \left(\mathrm{T}\, (\mathbf{r}^{t})\right) = 
				(r_s^t)^{d_s} \prod_{i=1}^{d_t}r_i^t + 
                \sum\limits_{i = 1}^{d_t} r_i^t(n_i - r_i^t) + d_s r_s^t (n_s - r_s^t).
			\end{align*}
   		\end{proof}
	
	\subsection{Tangent space}\label{sub:sfett_tangent_space}
		A tangent space provides a local linear approximation of a smooth manifold at a given point. It consists of tangent vectors, which can be characterized as the first derivatives of smooth curves on the manifold passing through that point:
		\begin{equation*}
			T_{\mathcal{X}} \mathrm{SFETT}\, (\mathbf{r}^{tt}, \mathbf{r}^{ts})= 
			\{ \gamma'(0) | \text{smooth curve }
				\gamma\colon \mathbb{R} \to\mathrm{SFETT}\, (\mathbf{r}^{tt}, \mathbf{r}^{ts}), 
			\gamma(0) = \mathcal{X}  \}.
		\end{equation*}
		The tangent space for the manifold of fixed SF-ETT rank at 
		\begin{equation*}
			\mathcal{X} = \left\llbracket 
			\mathcal{T}\left(W^{(1)}_L, \dotsc, W^{(d-1)}_L, \overline{W}^{(d)}\right); U^{(1)}, \dotsc, U^{(d_t)}, U, \dotsc, U
			\right\rrbracket
		\end{equation*}
		can be written as
		\begin{align*}
			T_{\mathcal{X}} &\mathrm{SFETT}\, (\mathbf{r}^{tt}, \mathbf{r}^{ts})  = \\
			&\left\{ 
				\xi  
				\bigg| \xi = 
				\left\llbracket
				\dot{\mathcal{T}}_{\mathcal{X}}\left(\dot{W}^{(1)},\dotsc,\dot{W}^{(d)}\right);
					U^{(1)}, \dotsc, U^{(d_t)}, U, \dotsc, U
				 \right\rrbracket + \right.\\ 
			&   \left. \left\llbracket\mathcal{T}\left(W^{(1)}_L, \dotsc, W^{(d-1)}_L, \overline{W}^{(d)}\right);
				{\dot{U}^{(1)}}, \dotsc, U^{(d_t)}, U, \dotsc, U\right\rrbracket + \dots + \right.\\
			& 	\left.\left\llbracket\mathcal{T}\left(W^{(1)}_L, \dotsc, W^{(d-1)}_L, \overline{W}^{(d)}\right);
				U^{(1)}, \dotsc, U^{(d_t)}, U, \dotsc, \dot{U}\right\rrbracket
			\right\},
		\end{align*}
		where 
		\begin{equation*}
			\dot{\mathcal{T}}_{\mathcal{X}}\left(\dot{W}^{(1)},\dotsc,\dot{W}^{(d)}\right) = 
			\sum_{i=1}^{d}\mathcal{T}\left(W^{(1)}_L,\dotsc,{\dot{W}}^{(i)},\dotsc,W^{(d)}_R\right),
		\end{equation*}
		$\dot{W}^{(i)} \in \mathbb{R}^{r^{\mathrm{tt}}_{i - 1}\times r^{\mathrm{t}}_{i}\times r^{\mathrm{tt}}_{i}}$, 
		$i \in \overline{1, d_t}$, $\dot{W}^{(i)} \in 
		\mathbb{R}^{r^{\mathrm{tt}}_{i - 1}\times r^{\mathrm{t}}_{s}\times r^{\mathrm{tt}}_{i}}$,
		$i \in \overline{d_t + 1, d}$
		($r^{\mathrm{tt}}_{0} = r^{\mathrm{tt}}_{d} = 1$), 
		$\dot{U}^{(j)} \in \mathbb{R}^{n_j \times r^{\mathrm{t}}_{j}}, j \in \overline{1,d_t}$, 
		$\dot{U}\in \mathbb{R}^{n_s \times  r^{\mathrm{t}}_{s}}$.
        Let $\mathcal{X}$ also be in the $(d-1)$ orthogonalized form. Then we impose additional
		gauge conditions:
    \begin{align*}\label{eq:sfett_tangent_gauge_cond}
			&\left(U^{(i)}\right)^\top\dot{U}^{(i)}  = O_{r^{\mathrm{t}}_{i}},\quad\forall i \in \overline{1, d_t},  \\
			&\left(U\right)^\top\dot{U}  = O_{r^{\mathrm{t}}_{s}},  \\
			&L \left( W^{(i)}_L \right)^\top L\left( \dot{W}^{(i)} \right) = O_{r^{\mathrm{tt}}_{i}}, \quad
			\forall i \in \overline{1, d - 1}. 
		\end{align*}
		It is also possible to store every tangent vector in SF-ETT form with at most doubled SF-ETT rank:
		\begin{equation*}\label{eq:sfett_tangent_elem_double_rank}
			\xi = \left\llbracket \mathcal{T}\left(\hat{W}^{(1)},\dotsc,\hat{W}^{(d)}\right),
			\begin{bmatrix} U^{(1)} & {\dot{U}^{(1)}} \end{bmatrix} , \dotsc, 
			\begin{bmatrix} U & {\dot{U}} \end{bmatrix}\right\rrbracket,
		\end{equation*}
		where $\mathcal{T}\left(\hat{W}^{(1)},\dotsc,\hat{W}^{(d)}\right) 
		\in \mathbb{R}^{2r_1\times \dotsc\times 2r_{d_t}\times 2r_{s}\times \dotsc\times 2r_s}$ and 
		\begin{align*}
			& \hat{W}^{(1)}\left[:r_1^t, :\right] = \begin{bmatrix} \dot{W}^{(1)} & W_L^{(1)} \end{bmatrix}, \quad
			\hat{W}^{(1)}[r_1^t:, :] = \begin{bmatrix} \overline{W}^{(1)} & 0 \end{bmatrix}; \\ 
		    & \hat{W}^{(i)}\left[:, :r_i^t, :\right] = \begin{bmatrix} W_R^{(i)} & 0 \\ \dot{W}^{(i)} & W_L^{(i)} \end{bmatrix}, \quad
			\hat{W}^{(i)}[:, r_i^t:, :] = \begin{bmatrix} 0 & 0 \\ \overline{W}^{(i)} & 0 \end{bmatrix}, 
			\quad i \in \overline{(1, d_t)};\\ 
		    & \hat{W}^{(j)}\left[:, :r_s^t, :\right] = \begin{bmatrix} W_R^{(j)} & 0 \\ \dot{W}^{(j)} & W_L^{(j)} \end{bmatrix}, \quad
			\hat{W}^{(j)}[:, r_s^t:, :] = \begin{bmatrix} 0 & 0 \\ \overline{W}^{(j)} & 0 \end{bmatrix}, 
			\quad j \in \overline{(d_t + 1, d)};\\ 
			& \hat{W}^{(d)}\left[:, :r_s^t\right] = \begin{bmatrix} W_R^{(d)} \\ \dot{W}^{(d)}  \end{bmatrix}, \quad
			\hat{W}^{(d)}[:, r_s^t:] = \begin{bmatrix} 0 \\ \overline{W}^{(d)}   \end{bmatrix}. 
		\end{align*}
		Where $\overline{W}^{(i)}, i \in \overline{1, d}$ is the $i$-th core in the $i$-th orthogonalization, where  
		$\overline{W}^{(i)}$ is the only core without the orthogonality constraint.
	\subsection{Projection to the tangent space}\label{sub:sfett_tangent_space_proj}
		The orthogonal projection to a tangent space of a smooth manifold is a crucial step in gradient-based 
		optimization algorithms on smooth manifolds. 
		This section proposes formulas for computing the orthogonal 
		projection onto the tangent space of $\mathrm{SFETT}\, (\mathbf{r}^{tt}, \mathbf{r}^{ts})$.

		Every orthogonal projection is associated with a scalar product, defined on tangent space, 
		called Riemannian metric. Further statements 
		are based on Riemannian metric coinciding with a standard scalar product in 
		$\mathbb{R}^{n_1\times \dotsc\times n_{d_t}\times n_{s}\times \dotsc\times n_s}$ on 
		$T_{\mathcal{X}} \mathrm{SFETT}\, (\mathbf{r}^{tt}, \mathbf{r}^{ts})$.
		In this case, the computation of Riemannian gradient can be simplified to 
		the computation of a projection of Euclidean gradient onto the tangent space:
		\begin{equation}\label{eq:riemannian_grad_as_proj}
			\mathrm{grad}f(\mathcal{X}) = P_{T_{\mathcal{X}} \mathrm{SFETT}\, (\mathbf{r}^{tt}, \mathbf{r}^{ts})} \left(\nabla \overline{f}(\mathcal{X})\right), 
		\end{equation}
        where $f\colon \mathrm{SFETT}\, (\mathbf{r}^{tt}, \mathbf{r}^{ts}) \to \mathbb{R}$ is a restriction of 
        some smooth function $\overline{f}\colon \mathbb{R}^{n_1 \times \dots n_{d_t} \times n_s \times \dots \times n_s} \to \mathbb{R}$ to $\mathrm{SFETT}\, (\mathbf{r}^{tt}, \mathbf{r}^{ts})$.
		Let $Z \in \mathbb{R}^{n_1\times \dotsc\times n_{d_t}\times n_{s}\times \dotsc\times n_s}$ 
		and $\mathcal{X} \in \mathrm{SFETT}\, (\mathbf{r}^{tt}, \mathbf{r}^{ts})$:
		\begin{equation*}
			\mathcal{X} = \left\llbracket \mathcal{G}; U^{(1)}, \dotsc, U^{(d_t)}, U, \dotsc, U\right\rrbracket, \quad 
			\mathcal{G} = \mathcal{T}\left(W^{(1)}, \dotsc,W^{(d-1)}, W^{(d)}\right),
		\end{equation*}
		then
		$P_{T_{\mathcal{X}} \mathrm{SFETT}\, (\mathbf{r}^{tt}, \mathbf{r}^{ts})} \left(Z \right) 
		\in T_{\mathcal{X}} \mathrm{SFETT}\, (\mathbf{r}^{tt}, \mathbf{r}^{ts})$  (or $P(Z)$ for short)
		admits the following SF-ETT decomposition:
		\begin{align*}
			P &\left(Z \right) = \\
			&\left\llbracket \dot{\mathcal{T}}_{\mathcal{X}}\left(\dot{W}^{(1)},\dotsc,\dot{W}^{(d)}\right);
			 U^{(1)}, \dotsc, U \right\rrbracket + 
			\left\llbracket \mathcal{G};
			{\dot{U}^{(1)}}, \dotsc, 
			 U\right\rrbracket+ \dots +
			\left\llbracket \mathcal{G}; U^{(1)}, \dotsc, \dot{U} \right\rrbracket,\\
			&\dot{U}^{(i)} 
			= P_{\perp}U^{(i)}
			\left\llbracket Z;{U^{(1)}}^\top,\dotsc,\underbrace{I}_{i},\dotsc,{U}^\top\right\rrbracket_{(i)}
			\mathcal{G}_{(i)}^\dagger,\quad\forall i \in \overline{1,d_t},\\
			&\dot{U} = 
			P_{\perp}U
			\left(\sum_{i=d_t+1}^{d}
			\left\llbracket Z;{U^{(1)}}^\top,\dotsc,\underbrace{I}_{i},\dotsc,{U}^\top\right\rrbracket_{(i)}
			\mathcal{G}_{(i)}^\top\right)
			\left(\sum_{i=d_t+1}^{d}\mathcal{G}_{(i)}\mathcal{G}_{(i)}^\top\right)^{-1},\\
			&\dot{W}^{(i)} 
			= 
			P_{\perp}L\left(W^{(i)}_L\right)
			\left(I_{r^{\mathrm{t}}_{i}} \otimes_K
			\mathcal{G}_{ \le i - 1}^\top\right)
			\hat{\mathcal{G}}^{<i>}
			\mathcal{G}_{ \ge i + 1},\quad\forall i \in \overline{1, d_t},\\
			&\dot{W}^{(i)} 
			= 
			P_{\perp}L\left(W^{(i)}_L\right)\left(I_{r^{\mathrm{t}}_{s}} \otimes_K
			\mathcal{G}_{ \le i - 1}^\top\right)
			\hat{\mathcal{G}}^{<i>}
			\mathcal{G}_{ \ge i + 1},\quad\forall i \in \overline{d_t + 1, d - 1},\\
			&\dot{W}^{(d)}
			= 
			\left(I_{r^{\mathrm{t}}_{s}} \otimes_K
			\mathcal{G}_{ \le d - 1}^\top\right) 
			\hat{\mathcal{G}}^{<d>},
		\end{align*}
		where
		\begin{equation*}
			P_{\perp}A = \left(I - AA^\top \right),\quad \hat{\mathcal{G}} = 
			\left\llbracket Z;{U^{(1)}}^\top,\dotsc,{U^{(d_t)}}^\top,{U}^\top,\dotsc,{U}^\top\right\rrbracket.
		\end{equation*}
		\begin{proof}
			The justification of the formulas above consists of two steps fully described in Appendix~\ref{appendix:tangent_space_proj}: the first stage is to notice that 
			tangent space can be decomposed into the sum of mutually orthogonal linear subspaces, then 
			deriving the projection stands behind deriving several projections onto these linear subspaces.
		\end{proof}	
	\subsection{Retraction}\label{sub:sfett_retraction}
		Retraction  
		\begin{equation}
			R(\cdot, \cdot)\colon \mathcal{M} \times T_{\mathcal{X}}\mathcal{M} \to \mathcal{M}		
		\end{equation}
 		for a smooth manifold
		$\mathcal{M}$ is a smooth mapping which satisfies the following two conditions \citep{absil2008optimization}:
		\begin{align}
			&R(\mathcal{X}, 0) = \mathcal{X}\label{eq:retraction_1},\\
			&\mathrm{D}_\xi R(\mathcal{X}, 0) [\eta] = \eta,\quad \forall \eta \in T_{\mathcal{X}}\mathcal{M}
			\label{eq:retraction_2}.
		\end{align}
		The retraction operation approximates exponential maps that are often difficult to compute from the computational viewpoint.
		The following theorem proposes the SF-ETT-Rounding Algorithm~\ref{alg:sfett_rounding} as a retraction for 
		$\mathrm{SFETT}\, (\mathbf{r}^{tt}, \mathbf{r}^{ts})$.
		\begin{proposition}\label{proposition:sfett_rounding_retraction}
			The mapping
			\begin{align*}
				R(\mathcal{X}, \xi) = P_{ (\mathbf{r}^{tt}, \mathbf{r}^{ts})}^{\mathtt{sf-ett-hosvd}} \left(\mathcal{X} + \xi\right),
			\end{align*}						
			where $P_{ (\mathbf{r}^{tt}, \mathbf{r}^{ts})}^{\mathtt{sf-ett-hosvd}}$ is from Theorem~\ref{theorem:sfett_hosvd} is a retraction for the manifold 
			 $\mathrm{SFETT}\, (\mathbf{r}^{tt}, \mathbf{r}^{ts})$.
		\end{proposition}
		\begin{proof}
			Using quasi-optimality of Algorithm~\ref{alg:sfett_rounding} and \cite[Proposition $2.3.$]{kressner2014low}, we immediately obtain the desired result.
		\end{proof}

\section{Gradient-based optimization on the manifold}\label{sec:gradient_based_optimization}
	This section is devoted to the description of  the key ingredients of the gradient-based Riemannian optimization algorithm on the manifold of the fixed SF-ETT rank. The details of their application in practice will be given further in Section~\ref{sec:num_exp}.
    
    Let a scalar-valued function $f\colon \mathrm{SFETT}\, (\mathbf{r}^{tt}, \mathbf{r}^{ts}) \to \mathbb{R}$ be a restriction of 
	 some smooth function $\overline{f}\colon~\mathbb{R}^{n_1 \times \dots \times n_{d_t}~\times n_s \times \dots \times n_s} \to \mathbb{R}$ to $\mathrm{SFETT}\, (\mathbf{r}^{tt}, \mathbf{r}^{ts})$. We aim to solve the task:
    
	\begin{equation}\label{eq:sfett_riem_task}
		\min\limits_{\mathcal{X} \in \mathrm{SFETT}\, (\mathbf{r}^{tt}, \mathbf{r}^{ts})} f(\mathcal{X}).
	\end{equation}

	A loop of the Riemannian gradient descent algorithm for solving \eqref{eq:sfett_riem_task} can be organized according to the following
    scheme \citep{absil2008optimization}. First, one computes an optimization direction $h_k$ parallel to the Riemannian anti-gradient $(-1)\cdot\mathrm{grad}f(\mathcal{X}_k)$ at the current point $\mathcal{X}_k$. Afterwards the direction is stretched by 
    a step size value $\alpha_k > 0$. 
    This optimization loop finalizes with 
    a step from $\mathcal{X}_k$ towards $\alpha_kh_k$ along $T_{\mathcal{X}_k} \mathrm{SFETT}\, (\mathbf{r}^{tt}, \mathbf{r}^{ts})$. The result of this 
    operation has to belong to $\mathrm{SFETT}\, (\mathbf{r}^{tt}, \mathbf{r}^{ts})$. Therefore one may use a retraction 
    operation to satisfy this condition:
    \begin{equation*}
		\mathcal{X}_{k + 1} = R(\mathcal{X}_{k},  \alpha_k h_k).
	\end{equation*}
    A schematic visualization of the presented optimization loop
    is shown in Figure~\ref{fig:notation_riem_opt}.
	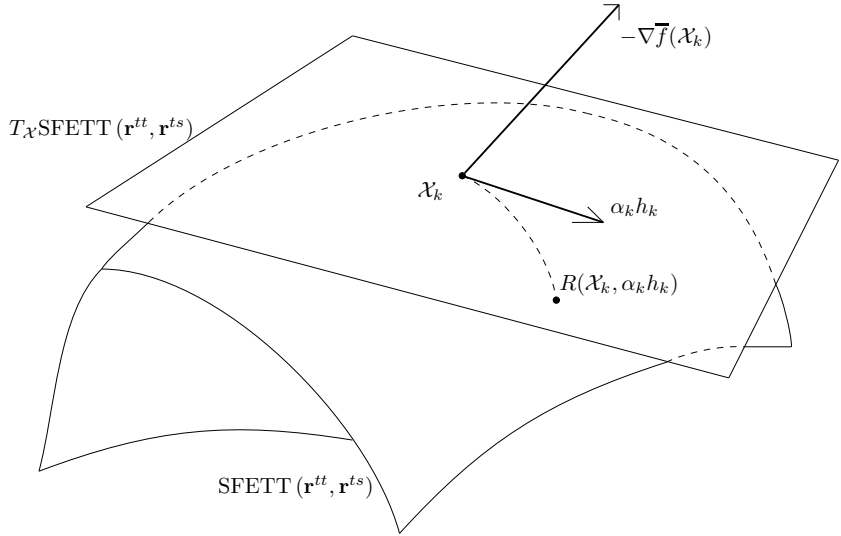
\begin{figure}[h!tp]
		\centering
		\resizebox{0.85\textwidth}{!}{
            \begin{tikzpicture}
            	\begin{pgfonlayer}{nodelayer}
            		\node [style=none] (0) at (-5.75, -2) {};
            		\node [style=none] (1) at (-4.75, 1.25) {};
            		\node [style=none] (2) at (0, -3) {};
            		\node [style=none] (3) at (3, 3.75) {};
            		\node [style=none] (4) at (-0.75, -1.5) {};
            		\node [style=none] (5) at (6.25, 0) {};
            		\node [style=none] (9) at (1, 2.75) {};
            		\node [style=none] (10) at (2.5, 0.75) {};
            		\node [style=none] (11) at (3.25, 2) {};
            		\node [style=none] (12) at (3, 2.25) {};
            		\node [style=none] (13) at (2.75, 2) {};
            		\node [style=none] (14) at (-5, 2.25) {};
            		\node [style=none] (15) at (5.25, -0.5) {};
            		\node [style=none] (16) at (-0.75, 5) {};
            		\node [style=none] (17) at (7, 3) {};
            		\node [style=none] (18) at (-4.75, 3.5) {$T_{\mathcal{X}} \mathrm{SFETT}\, (\mathbf{r}^{tt}, \mathbf{r}^{ts})$};
            		\node [style=none] (21) at (-1.65, -2.25) {$\mathrm{SFETT}\, (\mathbf{r}^{tt}, \mathbf{r}^{ts})$};
            		\node [style=none] (22) at (0.5, 2.5) {$\mathcal{X}_k$};
            		\node [style=none] (22) at (3.5, 1) {$R(\mathcal{X}_k, \alpha_k h_k)$};
            		\node [fill=black, draw=black, shape=circle, inner sep=0pt, minimum size=0.1cm] (23) at (1, 2.75) {};
            		\node [fill=black, draw=black, shape=circle, inner sep=0pt, minimum size=0.1cm] (24) at (2.5, 0.75) {};
            		\node [style=none] (25) at (-4, 2) {};
            		\node [style=none] (26) at (6, 1) {};
            		\node [style=none] (27) at (5.5, 0) {};
            		\node [style=none] (28) at (4.25, -0.25) {};
            		\node [style=none] (29) at (3.5, 5.5) {};
            		\node [style=none] (30) at (3.25, 5.5) {};
            		\node [style=none] (31) at (3.5, 5.25) {};
            		\node [style=none] (32) at (4.25, 5) {$-\nabla \overline{f}(\mathcal{X}_k)$};
            		\node [style=none] (33) at (3.75, 2.25) {$\alpha_k h_k$};
            		\node [style=none] (34) at (-2.75, 2) {};
            	\end{pgfonlayer}
            	\begin{pgfonlayer}{edgelayer}
            		\draw [in=-135, out=75, looseness=0.75] (0.center) to (1.center);
            		\draw [in=105, out=0, looseness=0.75] (1.center) to (2.center);
            		\draw [bend left=15] (0.center) to (4.center);
            		\draw [line width=0.03cm] (9.center) to (11.center);
            		\draw [bend left, looseness=0.75, dashed] (9.center) to (10.center);
            		\draw (11.center) to (12.center);
            		\draw (13.center) to (11.center);
            		\draw (14.center) to (15.center);
            		\draw (14.center) to (16.center);
            		\draw (16.center) to (17.center);
            		\draw (15.center) to (17.center);
            		\draw [in=-135, out=60, looseness=0.50] (1.center) to (25.center);
            		\draw [bend right=15, looseness=0.50] (5.center) to (26.center);
            		\draw [bend left=15] (2.center) to (28.center);
            		\draw [in=0, out=180] (5.center) to (27.center);
            		\draw [bend left=15, looseness=0.75, dashed] (28.center) to (27.center);
            		\draw [bend right, dashed] (26.center) to (3.center);
            		\draw [in=165, out=45, looseness=0.75, dashed] (25.center) to (3.center);
            		\draw [line width=0.03cm](23) to (29.center);
            		\draw (31.center) to (29.center);
            		\draw (30.center) to (29.center);
            	\end{pgfonlayer}
            \end{tikzpicture}
        }
		\caption{Visualization of the Riemannian gradient descent}
		\label{fig:notation_riem_opt}
	\end{figure}
	\subsection{Riemannian gradient computation}%
	\label{ssub:riemannian_gradient_computation}
		Let  
		\begin{equation*}
			\mathcal{X}_k = \left\llbracket \mathcal{T}\left(W^{(1)}_L, \dotsc, W^{(d-1)}_L, \overline{W}^{(d)}\right);
			U^{(1)}, \dotsc, U^{(d_t)}, U, \dotsc, U\right\rrbracket
		\end{equation*}
		be in $(d-1)$-orthogonal SF-ETT form. According to 
        \eqref{eq:riemannian_grad_as_proj},
		the Riemannian gradient computation can be reduced 
		to calculation of a projection of the Euclidean gradient $\nabla \overline{f}(\mathcal{X}_k)$ 
		to the tangent space:
		\begin{equation}\label{eq:riem_grad_as_proj}
			\mathrm{grad} f \left(\mathcal{X}_k \right) = 
			P_{T_{\mathcal{X}_k} \mathrm{SFETT}\, (\mathbf{r}^{tt}, \mathbf{r}^{ts})} 
			\nabla \overline{f}\left(\mathcal{X}_k \right).
		\end{equation}
		Since $\mathrm{grad} f \left(\mathcal{X}_k \right)\in T_{\mathcal{X}_k} \mathrm{SFETT}\, (\mathbf{r}^{tt}, \mathbf{r}^{ts})$, one can describe
		it  with tensors
		\begin{equation*}
			\dot{W}^{(i)}, i \in \overline{1, d}, \quad
			\dot{U}^{(j)}, j \in \overline{1,d_t}, \quad \dot{U},
		\end{equation*}
		that form a SF-ETT represented tensor with ranks which do not exceed $ (2\mathbf{r}^{tt}, 2\mathbf{r}^{ts})$  (see Section~\ref{sub:sfett_tangent_space}).

		Mentioning the fact that the Euclidean gradient may have a significantly bigger 
		SF-ETT rank than $(2\mathbf{r}^{tt}, 2\mathbf{r}^{ts})$ as described in \citep{novikov2022automatic} one 
		can come up with the idea of computing the whole projection in
		\eqref{eq:riem_grad_as_proj} at once. That can be accomplished via the automatic differentiation (autodiff) approach. This idea was firstly 
		described in \citep{novikov2022automatic}. Let us define a construction operator, which composes 
		the tangent vector as an element of $\mathbb{R}^{n_1 \times \dots n_{d_t} \times n_s \times \dots \times n_s}$ from its cores 
		and factors:
		\begin{align}
			&\mathcal{C}_{\mathcal{X}}\left(S^{(1)}, \dotsc, S^{(d)}; B^{(1)}, \dotsc, B^{(d_t)}; A\right) = \nonumber\\
			&\left\llbracket \dot{\mathcal{T}}_{\mathcal{X}}\left(S^{(1)},\dotsc,S^{(d)}\right);
			 U^{(1)}, \dotsc, U \right\rrbracket + 
			\left\llbracket \mathcal{G};
			{B^{(1)}}, \dotsc, 
			 U\right\rrbracket+ \dots +
			\left\llbracket \mathcal{G}; U^{(1)}, \dotsc, A \right\rrbracket.\label{eq:construction_operator}
		\end{align}
		One can easily verify that
		\begin{equation*}
			\mathcal{C}_{\mathcal{X}}\left(\underbrace{ 0, \dotsc, W^{(d)}; 0, \dotsc, 0; 0 }_{\theta}\right) = \mathcal{X},
		\end{equation*}
		and defining a new real-valued function $g$:
		\begin{equation*}
			g = \overline{f} \circ \mathcal{C}_{\mathcal{X}},
		\end{equation*}
		 one can compute partial derivatives of $g$ with respect to parameters 
		 $S^{(1)}, \dotsc, S^{(d)},$ $B^{(1)}, \dotsc, B^{(d_t)}, A$ via autodiff:
		\begin{align*}
			& \frac{\partial {\left(\overline{f} \circ \mathcal{C}_{\mathcal{X}} \right)}}
			{\partial {B^{(i)}}} \left( \theta \right) = 
			\left\llbracket \nabla \overline{f}\left(\mathcal{X}\right);
				{U^{(1)}}^\top,\dotsc,\underbrace{I}_{i},\dotsc,{U}^\top\right\rrbracket_{(i)}
			\mathcal{G}_{(i)}^\top,\quad\forall i \in \overline{1,d_t},\\
			& \frac{\partial {\left(\overline{f} \circ \mathcal{C}_{\mathcal{X}} \right)}}{\partial {A}} \left( \theta \right) = 
			\sum_{i=d_t+1}^{d}
			\left\llbracket \nabla \overline{f}\left(\mathcal{X}\right);
				{U^{(1)}}^\top,\dotsc,\underbrace{I}_{i},\dotsc,{U}^\top\right\rrbracket_{(i)}
			\mathcal{G}_{(i)}^\top, \\
			& \frac{\partial {\left(\overline{f} \circ \mathcal{C}_{\mathcal{X}} \right)}}
			{\partial {S^{(i)}}} \left( \theta \right) = 
			\left(I_{r^{\mathrm{t}}_{i}} \otimes_K
			\mathcal{G}_{ \le i - 1}^\top\right)
			\hat{\mathcal{G}}^{<i>}
			\mathcal{G}_{ \ge i + 1},\quad\forall i \in \overline{1, d_t},\\
			& \frac{\partial {\left(\overline{f} \circ \mathcal{C}_{\mathcal{X}} \right)}}
			{\partial {S^{(i)}}} \left( \theta \right) = 
			\left(I_{r^{\mathrm{t}}_{s}} \otimes_K
			\mathcal{G}_{ \le i - 1}^\top\right)
			\hat{\mathcal{G}}^{<i>}
			\mathcal{G}_{ \ge i + 1},\quad\forall i \in \overline{d_{t + 1}, d_s},\\
			& \frac{\partial {\left(\overline{f} \circ \mathcal{C}_{\mathcal{X}} \right)}}
			{\partial {S^{(d)}}} \left( \theta \right) = 
			\left(I_{r^{\mathrm{t}}_{s}} \otimes_K
			\mathcal{G}_{ \le i - 1}^\top\right)
			\hat{\mathcal{G}}^{<d>},
		\end{align*}
		where
		\begin{equation*}
			\hat{\mathcal{G}} =
			\left\llbracket \nabla \overline{f}(\mathcal{X});{U^{(1)}}^\top,\dotsc,{U^{(d_t)}}^\top,{U}^\top,\dotsc,{U}^\top\right\rrbracket,\quad
			\mathcal{G}= \mathcal{T}\left( W^{(1)}_L, \dotsc, W^{(d-1)}_L, \overline{W}^{(d)}\right).
		\end{equation*}
		\begin{proof}
			The justification of these formulas is presented in Appendix~\ref{appendix:riemannian_autodiff}.
		\end{proof}
		
		This leads to closed-form formulas for computing Riemannian gradient parameters on the tangent space:
		\begin{align*}
			&\dot{U}^{(i)} 
			= P_{\perp}U^{(i)}
			\frac{\partial {\left(\overline{f} \circ \mathcal{C}_{\mathcal{X}} \right)}}{\partial {B^{(i)}}} \left( \theta \right)
			\left(\mathcal{G}_{(i)}\mathcal{G}_{(i)}^\top\right)^{-1},\quad\forall i \in \overline{1,d_t},\\
			&\dot{U} 
			= P_{\perp}U
			\frac{\partial {\left(\overline{f} \circ \mathcal{C}_{\mathcal{X}} \right)}}{\partial {A}} \left( \theta \right)
			\left(\sum_{i=d_t + 1}^{d}\mathcal{G}_{(i)}\mathcal{G}_{(i)}^\top\right)^{-1},\\
			&\dot{W}^{(i)} 
			= 
			P_{\perp}L\left(W^{(i)}_L\right)
			\frac{\partial {\left(\overline{f} \circ \mathcal{C}_{\mathcal{X}} \right)}}{\partial {S^{(i)}}} \left( \theta \right)
			,\quad\forall i \in \overline{1, d - 1},\\
			&\dot{W}^{(d)}
			= 
			\frac{\partial {\left(\overline{f} \circ \mathcal{C}_{\mathcal{X}} \right)}}{\partial {S^{(d)}}} \left( \theta \right),\\
			&P_{\perp}A= \left(I - AA^\top \right).
		\end{align*}

    \begin{remark}
        In practice, the implementation of the construction operator from \eqref{eq:construction_operator} 
        does not perform explicit arithmetic operations. Instead, since the image of $C_{\mathcal{X}}$ belongs to the 
        tangent space $T_{\mathcal{X}} \mathrm{SFETT}\, (\mathbf{r}^{tt}, \mathbf{r}^{ts})$, 
        the tensor $ \mathcal{X} = C_{\mathcal{X}}(\theta)$ is stored in $(2\mathbf{r}^{tt}, 2\mathbf{r}^{ts})$ SF-ETT format as detailed in Section~\ref{sub:sfett_tangent_space}. Henceforth, $C_{\mathcal{X}}(\theta)$ is passed into the $\overline{f}$ as standard tensor in SF-ETT format.
    \end{remark}
    \subsection{Retraction}%
	\label{sub:retraction}
		As described in Section~\ref{sec:gradient_based_optimization}, once the search direction $h_k$ is given, the final step of a Riemannian optimization algorithm is 
		a retraction step:
		\begin{equation*}
			\mathcal{X}_{k + 1} = R(\mathcal{X}_{k},  \alpha_k h_k).
		\end{equation*}
		Proposition~\ref{proposition:sfett_rounding_retraction}
		states that Algorithm~\ref{alg:sfett_rounding} is a retraction for
        $\mathrm{SFETT}\,(\mathbf{r}^{tt}, \mathbf{r}^{ts})$:
		\begin{equation*}
			R(\mathcal{X}, \xi) = P_{(\mathbf{r}^{tt}, \mathbf{r}^{ts})}^{\mathtt{sf-ett-hosvd}} \left(\mathcal{X}  + \xi\right).
		\end{equation*}
		Since both tangent foot $\mathcal{X}_{k}$ and minimization direction $\alpha_k h_k$ belong to the same tangent space, its sum also belongs that linear space, so $(\mathcal{X}_{k} + \alpha_k h_k)$ is a SF-ETT decomposed tensor with at most $(2\mathbf{r}^{tt}, 2\mathbf{r}^{ts})$ SF-ETT rank
        that allows for straightforward application of Algorithm~\ref{alg:sfett_rounding} without any explicit calculations
        of dense tensors out of SF-ETT factors.
        
    \subsection{Momentum and vector transport}%
	\label{ssub:momentum_and_vector_transport}

    One can enhance the Riemannian gradient descent algorithm by additionally incorporating momentum vectors into the minimization direction. Specifically, the optimization direction $h_k$ can be seen as
    a linear combination of the current Riemannian gradient
	$\mathrm{grad}f(\mathcal{X}_k)$ and the 
	accumulation of Riemannian gradients from previous iterations $M_{k}$. Generally every new optimization point has 
    a new tangent space built around it, so each momentum vector $M_{k - 1} \in \mathrm{T}_{\mathcal{X}_{k-1}}\mathrm{SFETT} (\mathbf{r}^{tt}, \mathbf{r}^{ts})$ needs to be
	transported to the tangent space at the current point:
    \begin{equation*}
        M_k = \tau_k(M_{k - 1}) \in \mathrm{T}_{\mathcal{X}_{k}}\mathrm{SFETT} (\mathbf{r}^{tt}, \mathbf{r}^{ts}),
    \end{equation*}
	where $\tau_k \colon T_{\mathcal{X}_{k - 1}}
	\mathrm{SFETT}\, (\mathbf{r}^{tt}, \mathbf{r}^{ts}) \to 
	T_{\mathcal{X}_{k}} \mathrm{SFETT}\, (\mathbf{r}^{tt}, \mathbf{r}^{ts})$ is called a 
	\textit{vector transport function}~\citep{absil2008optimization}. 
	In particular, we use a projection onto the current 
	tangent space as a vector transport function:
	\begin{equation*}
		\tau_k \colon \mathrm{grad}f \left(\mathcal{X}_{ k - 1 } \right)  \mapsto
		P_{T_{\mathcal{X}_k}\mathrm{SFETT}\, (\mathbf{r}^{tt}, \mathbf{r}^{ts})} 
		\mathrm{grad}f \left(\mathcal{X}_{ k - 1 } \right).
	\end{equation*}
	The computation of such operation can be 
	reduced to one autodiff call as 
    proposed by differentiating
	\begin{equation*}
		\hat{f}(\mathcal{X}) = \langle \mathcal{X}, \xi \rangle, \quad \xi\in \mathbb{R}^{n_1\times \dotsc\times n_{d_t}\times n_{s}\times \dotsc\times n_s},
	\end{equation*}
	since
	\begin{equation*}
		\mathrm{grad}\hat{f}(\mathcal{X}) =
		P_{T_{\mathcal{X}} \mathrm{SFETT}\, (\mathbf{r}^{tt}, \mathbf{r}^{ts})} 
		 \nabla_{\mathcal{X}}\langle \mathcal{X}, \xi \rangle =
		P_{T_{\mathcal{X}} \mathrm{SFETT}\, (\mathbf{r}^{tt}, \mathbf{r}^{ts})} \xi.
	\end{equation*}
	Therefore, to compute a projection of any element $\xi$
	to a tangent space, one needs to compute a Riemannian gradient of the  
	function $\hat{f}$.

\section{Numerical experiments\label{sec:num_exp}}
    The developed Riemannian framework is highly versatile, enabling its use across a wide range of functionals and practical problems where tensor decompositions are applicable. In this section, we apply it to two examples: the compression of grid functions and the solution of multidimensional eigenvalue problems, and study the effect of using multiple shared factors.

    \subsection{Grid functions}%
    \label{sub:grid_functions}
    This section presents an evaluation of the performance of Algorithm~\ref{alg:sfett_rounding} in the context of approximating grid-based functions. The experimental procedure is structured as follows. Initially, a target function is approximated by a tensor $A$ 
    using the TT-cross approximation algorithm applied on a uniform grid
    $G = \{x_k | x_k = k /N, N \in \mathbb{R}\}$ \citep{oseledets2010tt} with high precision
    (e.g. $10^{-12}$ for \texttt{float32} precision).
    Subsequently, a fixed number of shared Tucker factors are initialized alongside regular Tucker factors, both set as identity matrices, and incorporated into the TT-cross approximated tensor. This yields an exact SF-ETT representation $\hat{A}$ of the tensor $A$, although the associated SF-ETT ranks may not be optimal.
    To obtain a quasi-optimal SF-ETT approximation $\mathcal{X}_0$,  Algorithm~\ref{alg:sfett_rounding} is applied with the prescribed SF-ETT ranks
    $(\mathbf{r}^{tt},\mathbf{r}^{ts})$, i.e  
    \begin{equation*}
        \mathcal{X}_0 = P^{\mathtt{sf-ett-hosvd}}_{(\mathbf{r}^{tt},\mathbf{r}^{ts})}\,(\hat{A}).
    \end{equation*}
    For the optimal approximation, $\mathcal{X}_0$ serves as the initial guess for a Riemannian steepest gradient descent (RStGD) method on the manifold
    $\mathrm{SFETT}\, (\mathbf{r}^{tt},\mathbf{r}^{ts})$ (see Section~\ref{sec:gradient_based_optimization})  aimed at minimizing the L$2$ error:
    
	\begin{equation*}
		\min\limits_{\mathcal{X} \in \mathrm{SFETT}\, (\mathbf{r}^{tt},\mathbf{r}^{ts})} \frac{1}{2}\| A - \mathcal{X} \|_F^2.
	\end{equation*}
	The Riemannian gradient of this objective is given by
	\begin{equation*}
		\mathrm{grad} f(\mathcal{X}) =  P_{T_{\mathcal{X}} \mathrm{SFETT}\, (\mathbf{r}^{tt}, \mathbf{r}^{ts})}
		\left( A - \mathcal{X} \right).
	\end{equation*}
	The steepest optimization step $\alpha_k^\star$ in 
	$\mathcal{X}_{k + 1} = \mathcal{X}_k - \alpha_k \, \mathrm{grad} f(\mathcal{X}_k)$ is 
	\begin{equation*}
		\alpha_k^\star = - \frac{\langle\mathrm(A - \mathcal{X}_k), \mathrm{grad} f(\mathcal{X}_k)  \rangle}
		{\|\mathrm{grad} f(\mathcal{X}_k)\|^2_F}.
	\end{equation*}
    The RStGD procedure is stopped unless the next optimization point gives a smaller objective function value. 
    
    We apply the proposed approach to approximate  $12$-dimensional function defined as:
    \begin{equation*}
        f(x) = 1 / (1 + c^\top x), \quad c_i = i + 1, \quad i \in \overline{1,12}, \quad x \in \mathbb{R}^{12},
    \end{equation*}
	on $12$ independent uniform grids $G_i$ each consisting of $N_i = 512$ points:
    \begin{equation*}
	   G_i = \{x_k | x_k  = k / 512,   k \in \overline{1, 512}\}, \quad i \in \overline{1, 12}.
    \end{equation*}
    Hence, the tensor $\mathcal{X}_0$ is a $12$-dimension tensor
    with mode sizes $n_i = 512$ for $i \in \overline{1, 12}$. For varying numbers of shared factors, we construct $\mathcal{X}_0$ with different rank configurations and record the number of parameters in each tensor. Subsequently, we minimize the L$2$
    approximation error and depict the relationship between the number of parameters and the 
    error in  Figure~\ref{fig:hilbert_tens}. The steps-like morphology of the curves is a result of the iterative procedure, wherein the Tucker rank and the Tensor Train (TT) rank are increased alternately by one. Table~\ref{table:min_max_rel_diff_l2_err_hilbert}
    demonstrates relative differences between relative L$2$ error without using RStGD and 
    with it.
    \begin{figure}[h!tp]
		\centering
		\includegraphics[width=0.9\linewidth]{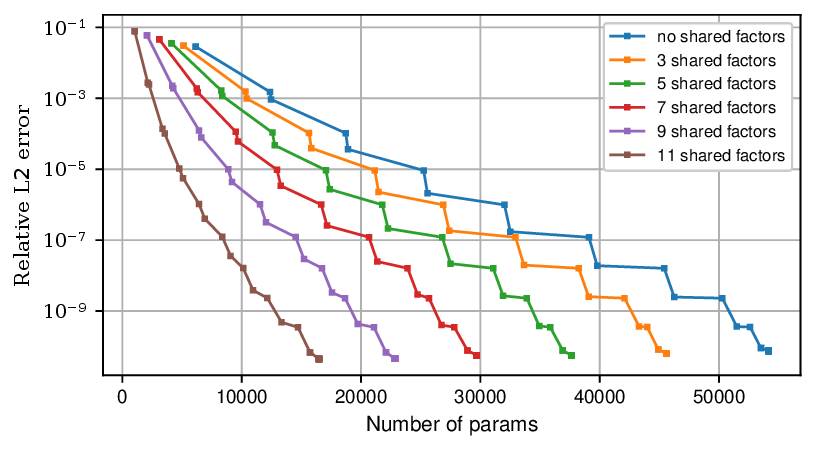}
		\caption{The SF-ETT with RStGD approximation L$2$ error vs. number of parameters of a $12$-tensor of $f(x) = 1/(c_1x_1 + \dotsm + c_{12}x_{12}), c_i = i + 1$ values 
			on a uniform grid of $512$ points for every argument with different amount of shared cores}
		\label{fig:hilbert_tens}
	\end{figure}
    \begin{table}[h!tp]
        \centering
        \caption{Minimum and maximum relative difference in L$2$ errors among different number of parameters for each number of shared factors between SF-ETT approximation of a $12$-tensor of $f(x) = 1/(c_1x_1 + \dotsm + c_{12}x_{12})$, $c_i = i + 1$ values 
        on a uniform grid of $512$ points for every argument with and without RStGD}
        \scriptsize
        \setlength{\tabcolsep}{4pt}
        \begin{tabular}{lcc}
            Shared Factors & Min Relative Difference & Max Relative Difference \\
            \hline
            no sharing & $9.48\times 10^{-7}$ & $9.84\times 10^{-4}$ \\
            3 & $3.97\times 10^{-6}$ & $3.59\times 10^{-3}$ \\
            5 & $9.47\times 10^{-6}$ & $3.66\times 10^{-2}$ \\
            7 & $1.85\times 10^{-5}$ & $2.15\times 10^{-3}$ \\
            9 & $1.64\times 10^{-6}$ & $1.23\times 10^{-2}$ \\
            11 & $4.81\times 10^{-5}$ & $5.25\times 10^{-3}$ \\
            \hline
        \end{tabular}
        \label{table:min_max_rel_diff_l2_err_hilbert}
    \end{table}

	The second function considered for approximation is  $f(x) = \exp(-\alpha x^2), x \in [0, 1]$ sampled on a uniform grid with step size  $1 / 32^{10}$. Accordingly, 
	the tensor $A$ is constructed in the QTT format \citep{approxmattt,khoromskij2011d} with mode sizes
	$n_i = 32, i \in \overline{1, 10}$. 
	The performance of the algorithm for different numbers of shared Tucker factors is summarized in
	Figure~\ref{fig:expontent}. This steps-like morphology shares its origin with that of Figure~\ref{fig:hilbert_tens}, stemming from the same iterative rank-incrementing process. Table~\ref{table:min_max_rel_diff_l2_err_exp}
    demonstrates relative differences between relative L$2$ error without using RStGD and 
    with it.
   	\begin{figure}[h!tp]
		\centering
		\includegraphics[width=0.9\linewidth]{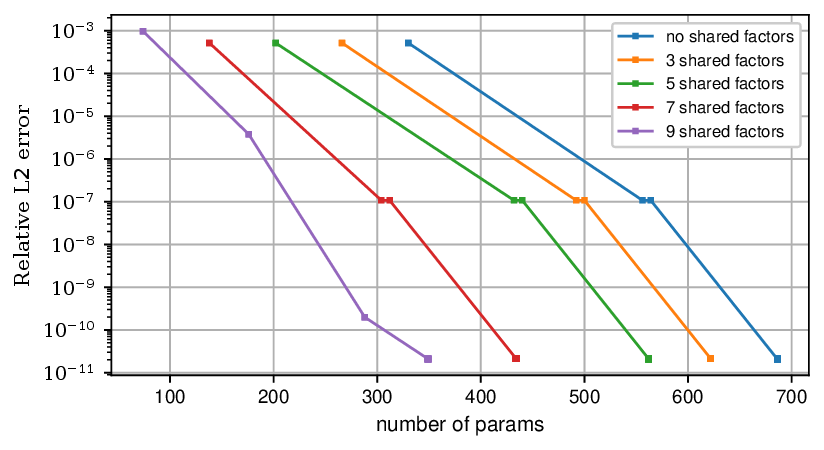}
		\caption{The SF-ETT with RStGD approximation L$2$ error vs. number of parameters of a $10$-tensor of $f(x) = \exp(-\alpha x^2), \alpha = 0.1$ values on a uniform grid of $32^{10}$ nodes 
			 with different amount of shared cores}
		\label{fig:expontent}
	\end{figure}
    \begin{table}[h!tp]
        \centering
        \caption{Minimum and maximum relative difference in L$2$ errors among different number of parameters for each number of shared factors between SF-ETT approximation of $f(x) = \exp(-\alpha x^2), \alpha = 0.1$ values on a uniform grid of $32^{10}$ nodes with and without RStGD}
        \scriptsize
        \setlength{\tabcolsep}{4pt}
        \begin{tabular}{lcc}
            Shared Factors & Min Relative Difference & Max Relative Difference \\
            \hline
            no sharing & $4.28\times 10^{-13}$ & $3.73\times 10^{-6}$ \\
            3 & $4.28\times 10^{-13}$ & $3.71\times 10^{-6}$ \\
            5 & $4.29\times 10^{-13}$ & $3.83\times 10^{-6}$ \\
            7 & $4.27\times 10^{-13}$ & $1.95\times 10^{-6}$ \\
            9 & $2.25\times 10^{-11}$ & $9.23\times 10^{-6}$ \\
            \hline
        \end{tabular}
        \label{table:min_max_rel_diff_l2_err_exp}
    \end{table}

	These experiments demonstrate that despite the quasi-optimal nature of Algorithm~\ref{alg:sfett_rounding}, the returned approximation closely approaches local optima in terms of the Frobenius norm. Furthermore, the subsequent steepest gradient descent step yields only marginal improvements in relative error in both tested configurations.
	
	\subsection{Eigenvalue problem}%
	\label{sub:eigenvalue_problem}

    This section presents numerical solutions for the eigenvalue problem of symmetric positive definite matrices, constrained to the manifold of fixed SF-ETT rank. We consider two test cases: a discretized Laplace operator with and without the Henon-Heiles potential matrix \citep{dolgov2014computation}.   Our approach utilizes a Riemannian optimization framework. Similar and other Riemannian 
	approaches for the eigenproblem can be found in \citep{baker2008riemannian,rakhuba2019low,rakhuba2018jacobi,krumnow2021computing}.

    Computing the extreme eigenvectors and eigenvalues of a symmetric positive definite matrix $H$ is equivalent to minimizing the Rayleigh quotient:
    \begin{equation*}
        \min_{\mathcal{X} \in \mathbb{R}^{n_1 \times  \dotsm \times n_d}} R(\mathcal{X} ), \quad R(\mathcal{X} ) = \frac{\langle  \mathrm{vec}\,(\mathcal{X}) , H \mathrm{vec}\,(\mathcal{X}) \rangle}{\langle \mathrm{vec}\,(\mathcal{X}) , \mathrm{vec}\,(\mathcal{X})  \rangle}, \quad H^\top = H.
    \end{equation*}
    We employ the Locally Optimal Block Preconditioned Conjugate Gradient (LOBPCG) algorithm \citep{knyazev2001toward} for a single vector.
    In the SF-ETT case we replace the whole search space $\mathbb{R}^{n_1 \times  \dotsm \times n_d}$ with the SF-ETT manifold of fixed rank.
    
    The LOBPCG algorithm is a gradient-based method for solving eigenvalue problems. Unlike simple gradient descent, it leverages not only the gradient at the current iterate but also information from gradients computed in previous iterations. At each step, the algorithm constructs a search subspace spanned by the current eigenvector approximation, the current gradient, and a set of accumulated search directions. A locally optimal solution within this subspace is then found using the Rayleigh-Ritz procedure $\mathtt{RayleighRitz}\left(H,S\right)$. In particular,  let $S$ be a basis for a search space. Then we need to solve the following generalized eigenvalue problem:
    \[
     (S^\top H S)Z = (S^\top S)Z \theta.
    \]
    Although LOBPCG is capable of computing multiple eigenpairs simultaneously (block method), in the present experiments it is configured to calculate only the eigenpair corresponding to the smallest eigenvalue.
    

    
	We use a Riemannian version of the Locally Optimal Conjugate Gradient (LOCG) Algorithm~\ref{alg:riemannian_sfett_lobpcg} that is constrained to the manifold of fixed SF-ETT rank. Preconditioning and block versions could also be included, but we restrict ourselves to the basic case to isolate the effect of sharing. Algorithm~\ref{alg:riemannian_sfett_lobpcg} approximates the 
    unit eigenvector, which corresponds to the smallest eigenvalue. In practice, this approximation is stored as $(d-1)$-orthogonal SF-ETT tensor, where $d$ is the amount of TT cores and Tucker factors.
	\begin{algorithm}[tp]
		\caption{Riemannian SF-ETT LOCG}
		\label{alg:riemannian_sfett_lobpcg}
		\textbf{Require:}
		\begin{enumerate}
			\item[]$H$ --- positive definite TT matrix,
			\item[]$\mathcal{X} = \left\llbracket \mathcal{T}\left( W^{(1)}, \dotsc, W^{(d)}\right); 
				U^{(1)}, \dotsc, U^{(d_t)}, U, \dotsc, U\right\rrbracket$ --- SF-ETT representation,
			\item[]$\left(\mathbf{r}^{\mathtt{tt}}, \mathbf{r}^{\mathtt{ts}}\right)$ --- the SF-ETT rank of $\mathcal{X}$.
		\end{enumerate}
		\textbf{Ensure:}
		\begin{enumerate}
			\item[]  $(\theta, \hat{\mathcal{X}})$ --- eigen pair for $H$ with minimal eigenvalue.
		\end{enumerate}
		\textbf{Function:}
		\begin{algorithmic}[1] 
			\State $[c_{0}, \theta_{0}] := \mathtt{RayleighRitz}\left(H,\mathcal{X}\right)$.
			\State $\mathcal{X}_{0} := c_{0}\mathcal{X}$.
			\State $R_{0} := P_{{T_{ \mathcal{X}_{0}}}{\mathrm{SFETT}\, (\mathbf{r}^{\mathrm{tt}},\mathbf{r}^{\mathrm{ts}})}}
			H \mathcal{X}_{0} - 
			\theta_{0}[1,1]\mathcal{X}_{0} $.
			\State $P_{0} := []$.
			\For {$k := 0, 1, \dots, \mathtt{max\_iters}$}
    			\State $S_{k} := [\mathcal{X}_{k}, R_{k}, P_{k}]$.
    			\State $[C_{k + 1}, \theta_{k + 1}] := \mathtt{RayleighRitz}\left(H, S_{k}\right)$.
    			\State $\mathcal{X}_{k + 1} := S_{k}C_{k+1}[:, 1]$.
    			\State $\mathcal{X}_{k + 1} := \mathtt{SF}\text{-}\mathtt{ETTRound}\left(\mathcal{X}_{k + 1},
    			\left(\mathbf{r}^{\mathrm{tt}},\mathbf{r}^{\mathrm{ts}}\right)\right)$.
    			\State $\mathcal{X}_{k + 1} := \mathcal{X}_{k+1} / \mathtt{norm}\left(\mathcal{X}_{k+1}, '\mathtt{F}'\right).$
    			\State $R_{k + 1} :=P_{{T_{\mathcal{X}_{k + 1}}}{\mathrm{SFETT}\, \left(\mathbf{r}^{\mathrm{tt}},\mathbf{r}^{\mathrm{t}}\right)}} 
    			L \mathcal{X}_{k + 1}  -  \theta_{k + 1}[1,1]\mathcal{X}_{k + 1}$.
    			\State $P_{k + 1} := S_{k + 1}[:, 2:]C^{ (k+1)}[2:, 1]$.
    			\State $P_{k + 1} := P_{{T_{\mathcal{X}_{k + 1}}}{\mathrm{SFETT}\, \left(\mathbf{r}^{\mathrm{tt}},\mathbf{r}^{\mathrm{ts}}\right)}} 
    			 P_{k + 1}$.
			\EndFor
			\State \Return $\left(\theta_{k + 1}[1,1], \mathcal{X}_{ k + 1 }\right)$.
		\end{algorithmic}
	\end{algorithm}

	 \subsubsection{Laplace operator}%
	 \label{ssub:laplace_operator}
	 
 	 The eigenproblem for the discretized Laplace operator on a uniform grid, which arises when searching for the states of an isolated particle in a  
	 $d$-dimensional box, can be written as follows:
	 \begin{equation*}
	  - L X = \Theta X, \quad L = D \otimes_K I_n \otimes_K \dotsc \otimes_K I_n + \dotsc  + I_n \otimes_K \dotsc \otimes_K I_n \otimes_K D,
	 \end{equation*}
	 where $I_n \in \mathbb{R}^{n \times n}$ is an identity matrix and $D = \mathrm{tridiag}\, \left( 1, -2, 1 \right) \in  \mathbb{R}^{n \times n}$. The tensor $L$ has a maximum TT rank bounded by $2$ \citep[Proposition~$4.1$]{dmrgqtt}.

	The analytical solution to this problem is known in advance, so the task primarily serves as a benchmark for eigensolvers. The eigenvalues of the matrix $(-D)$ are formed in pairs ${\theta_l, X_l}$:
	\begin{equation*}
		\theta_l = 4\sin^2 \left( \frac{\pi (l + 1)}{2(n + 1)} \right), \quad 
		X_{l,j} = \sin \left( \frac{\pi (l + 1)(j + 1)}{(n + 1)} \right), \quad l,j \in \overline{0, n-1}.
	\end{equation*}

	The performance of Algorithm~\ref{alg:riemannian_sfett_lobpcg} for the Laplace operator with different amount of shared Tucker 
	factors and different ranks
	are pictured in Figure~\ref{fig:laplace_rel_err}.
	Figure~\ref{fig:time_to_iters_laplace_no_pot} demonstrates the reduction in computation time achieved by Algorithm~\ref{alg:riemannian_sfett_lobpcg} with an increasing number of shared Tucker factors.

	\begin{figure}[h!tp]
		\centering
		\includegraphics[width=0.9\linewidth]{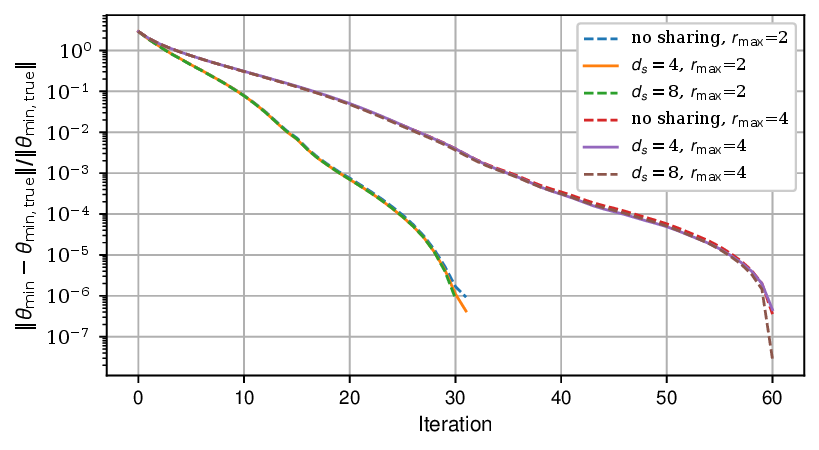}
		\caption{Relative error of the minimal eigenvalue approximation via Algorithm~\ref{alg:riemannian_sfett_lobpcg}
			for discretized Laplace operator of dimension $d = 8$,
		mode size $n = 32$ on the manifold of fixed SF-ETT rank for different ranks
        and shared Tucker factors amount}
		\label{fig:laplace_rel_err}
	\end{figure}
	\begin{figure}[h!tp]
		\centering
		\includegraphics[width=0.9\linewidth]{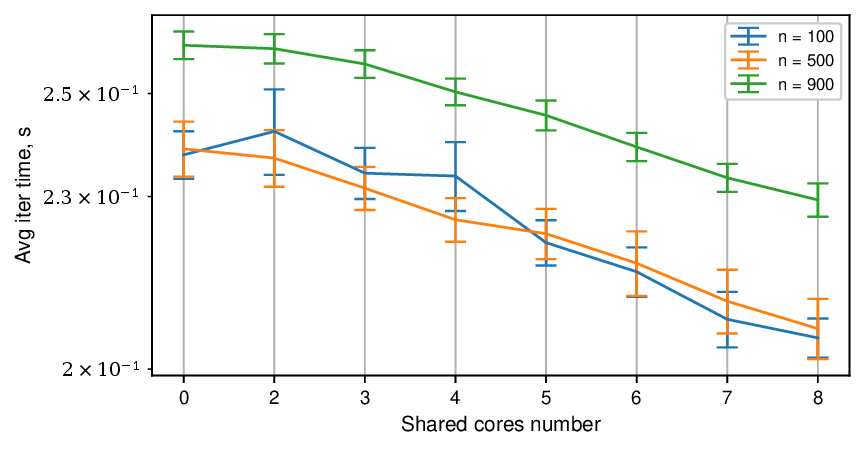}
		\caption{The average time (in seconds) of one loop iteration in Algorithm~\ref{alg:riemannian_sfett_lobpcg} for
		discretized Laplace operator of dimension $d = 8$ on the manifold of fixed SF-ETT rank for 
		different mode size $n$ and amount of shared Tucker factors}
		\label{fig:time_to_iters_laplace_no_pot}
	\end{figure}

	\subsubsection{Henon-Heiles}%
	\label{ssub:henon_heiles}
	
	Another setup is an eigenproblem given by the following matrix
	\begin{equation}\label{eq:henon_heiles_fun}
		H =  L  + \mathrm{diag}(\hat{V}),
	\end{equation}
	where $L$ is the discretized Laplace operator as in Section~\ref{ssub:laplace_operator} and $\hat{V}$ denotes a tensor of discretized Henon-Heiles potential function $V$ on a uniform grid:
	\begin{equation}\label{eq:henon_heiles_pot}
		V(x_1, \dotsc, x_m) = \frac{1}{2}\sum_{k=1}^{m} x_k^2 + \lambda \sum_{k=1}^{m-1} \left( x_k^2 x_{k + 1} - \frac{1}{3}x_{k + 1}^3 \right),
		\quad  \lambda = 0.111803.
	\end{equation}
	The tensor $\hat{V}$, representing the Henon-Heiles potential, has a maximum TT rank bounded by 3 \citep[Lemma~$4.10$]{dmrgqtt}. To get a proper SF-ETT approximation of the Henon-Heiles potential we utilize the same approach as in  
	Section~\ref{sub:grid_functions} without RStGD procedure. One may notice that 
	utilizing of such approximation for \eqref{eq:henon_heiles_fun} depends on the number of shared Tucker factors since
	SF-ETT approximation of the Henon-Heiles potential \eqref{eq:henon_heiles_pot} depends on it. 

	The experiment consists of approximating the Henon-Heiles potential \eqref{eq:henon_heiles_pot} and constructing the bilinear function \eqref{eq:henon_heiles_fun}, subsequently computing its minimal eigenpair using Algorithm~\ref{alg:riemannian_sfett_lobpcg}. For validation, this eigenpair was also calculated using the direct \texttt{eigh} solver from the \texttt{ttpy} package \citep{ttpy2012}, which serves as a reference.

	The performance of Algorithm~\ref{alg:riemannian_sfett_lobpcg} applied to \eqref{eq:henon_heiles_fun} is shown in Figure~\ref{fig:hh_rel_err} for varying numbers of shared Tucker factors and different ranks. Figure~\ref{fig:time_to_iters_laplace_hh} demonstrates the corresponding reduction in iteration time as the number of shared Tucker factors increases.  As the computations for \eqref{eq:henon_heiles_fun} are more complex, the increased number of shared factors yields a more notable speedup than in the case of the simple Laplace operator (see Figure~\ref{fig:time_to_iters_laplace_no_pot}).
	\begin{figure}[h!tp]
		\centering
		\includegraphics[width=0.9\linewidth]{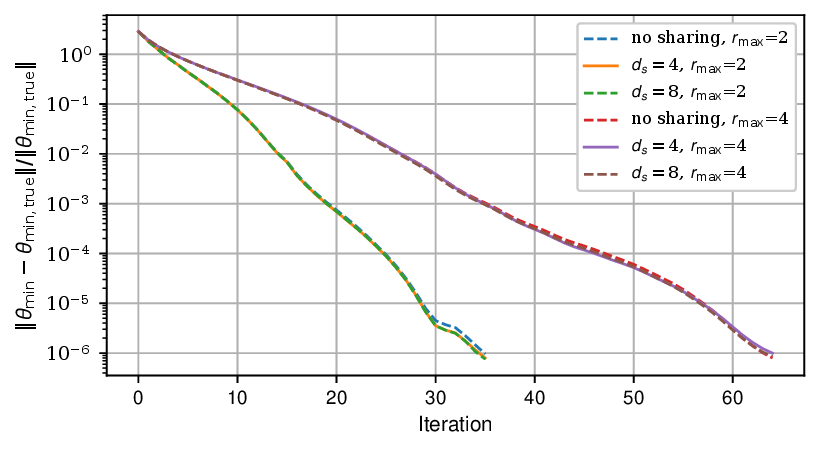}
		\caption{Relative error of the minimal eigenvalue approximation via Algorithm~\ref{alg:riemannian_sfett_lobpcg}
			for discretized function in \eqref{eq:henon_heiles_fun} function of dimension $d = 8$,
		mode size $n = 32$ on the manifold of fixed SF-ETT rank for different ranks
        and shared Tucker factors amount}
		\label{fig:hh_rel_err}
	\end{figure}

	\begin{figure}[h!tp]
		\centering
		\includegraphics[width=0.9\linewidth]{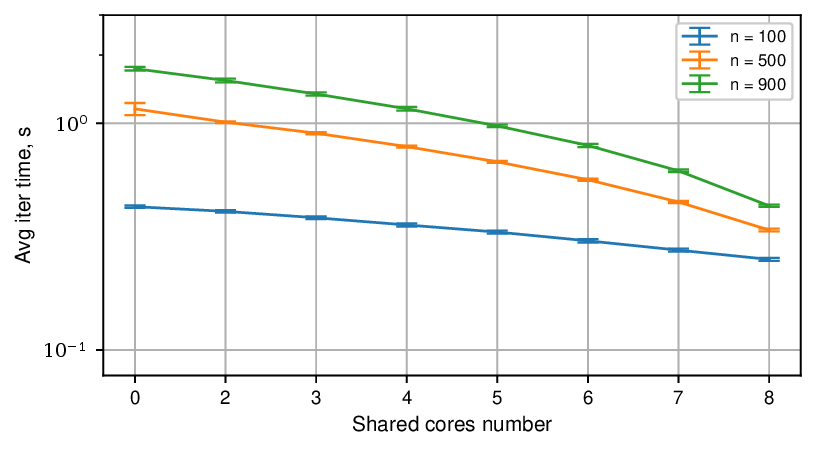}
		\caption{The average time (in seconds) of one loop iteration in Algorithm~\ref{alg:riemannian_sfett_lobpcg} for 
			function in \eqref{eq:henon_heiles_fun} function of dimension $d = 8$ on the manifold of fixed SF-ETT rank for 
			different mode size $n$ and amount of shared Tucker factors}
		\label{fig:time_to_iters_laplace_hh}
	\end{figure}

\section{Conclusion}\label{sec:conclusion}

We explore a Riemannian optimization approach for an ETT decomposition with an additional constraint that enforces equality of the factors. We develop a general framework, along with its software implementation, that provides the necessary components for Riemannian optimization algorithms. This framework can be readily applied to achieve further memory and time reduction where the classical ETT decomposition is of interest. 

\bmhead{Acknowledgements}

The authors are thankful to Danil Kolyadin for valuable consultations concerning the rounding algorithm.
\begin{appendices}
\section{Projection on tangent space of SF-ETT manifold}\label{appendix:tangent_space_proj}
	\begin{proposition}[Tangent space structure]\label{state:tangent_space_structure}
		Let $\mathcal{X} \in \mathrm{SFETT}\, (\mathbf{r}^{tt}, \mathbf{r}^{ts})$
		is a tangent foot with SF-ETT decomposition
		\begin{equation*}
		\mathcal{X} = \left\llbracket \mathcal{G}; V^{(1)}, \dotsc, V^{(d_t)}, U, \dotsc, U\right\rrbracket, \quad 
		\mathcal{G} = \mathcal{T}\left(W^{(1)}, \dotsc, W^{(d)}\right),
		\end{equation*}
		then the tangent space $T_{\mathcal{X}}\mathrm{SFETT}$ has the following structure
		\begin{equation*}
			\left(\bigoplus_{i = 1}^d L_0^{(i)}\right) \oplus  L_1  \oplus \dotsc \oplus L_{d_t} \oplus L_{d_s},
		\end{equation*}
		where 
		\begin{align*}
			&L_0^{(i)} = \\
			& \left\{ \xi \bigg| 
			\xi = \left\llbracket 
				\mathcal{T}\left(W_L^{(1)}, \dots, \dot{W}^{(i)}, \dotsc, W^{(d)}_R\right); 
				V^{(1)}, \dotsc,  U\right\rrbracket,
				L\left(\dot{W}^{(i)}\right)^\top L\left(W_L^{(i)}\right) = 0
				\right\}, \\
			&i \in \overline{1, d - 1}, \\
			&L_0^{(d)} = 
			 \left\{ \xi \bigg| 
			\xi = \left\llbracket 
				\mathcal{T}\left(W_L^{(1)}, \dots, W_L^{(d-1)}, \dot{W}^{(d)}\right); 
				V^{(1)}, \dotsc,  U\right\rrbracket
				\right\}, \\
			&L_i = 
			 \left\{ \xi \bigg| 
			\xi = \left\llbracket 
				\mathcal{G}; 
				V^{(1)}, \dotsc,  \dot{V}^{(i)}, \dotsc, V^{(d_t)}, U, \dotsc, U\right\rrbracket,
				\left( \dot{V}^{(i)} \right)^{\top} V^{(i)} = 0
				\right\}, \quad i \in \overline{1, d_t}, \\
			&L_s = 
			 \left\{ \xi \Bigg| 
			\xi = \sum_{i = d_t + 1}^{d} \left\llbracket 
				\mathcal{G}; 
				V^{(1)}, \dotsc,   V^{(d_t)}, U, \dotsc, \underbrace{\dot{U}}_i, \dotsc, U\right\rrbracket,
			 \dot{U} ^{\top} U= 0\right\}.
		\end{align*}
	\end{proposition}
	\begin{proof}
		Let $\xi \in L_{0}^{(i)}, \eta \in L_j, \forall i \in \overline{1,d}, j \in \overline{1,d_t}$ then
		\begin{align*}
			\left\langle \xi, \eta \right\rangle &=  \\ 
			 & \left\langle 
				\left\llbracket \dot{\mathcal{G}}; 
				V^{(1)}, \dotsc, V^{(d_t)}, U, \dots, U\right\rrbracket, 
						\left\llbracket 
				\mathcal{G}; V^{(1)}, \dotsc, \dot{V}^{(j)}, \dotsc, V^{(d_t)}, U, \dots, U
			\right\rrbracket\right\rangle = 0,
		\end{align*}
		due to the gauge condition
		\begin{equation*}
			\left( \dot{V}^{(j)} \right)^\top V^{(j)} = 0.
		\end{equation*}

		Now let $\eta \in L_s$ then
		\begin{align*}
			&\left\langle \xi, \eta \right\rangle =  \\ 
			 & \sum_{i=1}^{d_s} \left\langle 
				\left\llbracket \dot{\mathcal{G}}; 
				V^{(1)}, \dotsc, V^{(d_t)}, U, \dots, U\right\rrbracket, 
						\left\llbracket 
				\mathcal{G}; V^{(1)}, \dotsc,  V^{(d_t)}, U, \dots, \underbrace{\dot{U}}_i, \dots, U
			\right\rrbracket\right\rangle = 0,
		\end{align*}
		due to the gauge condition
		\begin{equation*}
			\left( \dot{U} \right)^\top U^{(j)} = 0.
		\end{equation*}
		
		Finally, let $\xi \in L_0^{(i')}, i \neq i'$ then
		\begin{align*}
			&\left\langle \xi, \eta \right\rangle =  \\ 
			 & \left\langle 
				 \left\llbracket \dot{\mathcal{T}}_{\mathcal{X}} \left( 0, \dotsc,\dot{W}^{(i)},\dotsc,0\right); 
				V^{(1)}, \dotsc, V^{(d_t)}, U, \dots, U\right\rrbracket, \right.\\ 
			 &\left.\left\llbracket 
				\dot{\mathcal{T}}_{\mathcal{X}}\left( 0, \dotsc,\dot{W}^{(i')},\dotsc,0\right); V^{(1)}, \dotsc,  V^{(d_t)}, U, \dots,  U
			\right\rrbracket\right\rangle = \\
			 & \left\langle 
				 \left\llbracket \dot{\mathcal{T}}_{\mathcal{X}} \left( 0, \dotsc,\dot{W}^{(i)},\dotsc,0\right); 
				I, \dotsc, I, I, \dots, I\right\rrbracket, \right.
			 \left.
				\dot{\mathcal{T}}_{\mathcal{X}}\left( 0, \dotsc,\dot{W}^{(i')},\dotsc,0\right)
			\right\rangle = \\
			 & \left\langle  
			 \mathcal{T}\left(W_L^{(1)}, \dots, \dot{W}^{(i)}, \dots,  W_R^{(d)}\right),
			 \mathcal{T}\left(W_L^{(1)}, \dots, \dot{W}^{(i')}, \dots,  W_R^{(d)}\right)
		 \right\rangle = 0,
		\end{align*}
		due to the gauge condition
		\begin{equation*}
			L\left( W_L^{(i)} \right)^\top  L\left( \dot{W}^{(i)} \right)= 0.
		\end{equation*}
		Thus $\forall i \in \overline{1,d}, i' \in \overline{1,d}, i' \neq i, \forall j \in \overline{1,d_t}$ it is implied that 
		$L_0^{(i)}\perp L_0^{(i')}$, $L_0^{(i)}\perp L_j$, $L_0^{(i)}\perp L_s$. The orthogonality of $L_j$ against $L_{j'}$ and
		against $L_s$ is the same as in \citep{peshekhonov2024training}.
	\end{proof}

	\begin{proposition}[Tangent space projection]
		Let $Z \in \mathbb{R}^{n_1\times \dotsc\times n_{d_t}\times n_{s}\times \dotsc\times n_s}$ 
		and $\mathcal{X} \in \mathrm{SFETT}\, (\mathbf{r}^{tt}, \mathbf{r}^{ts})$
		is a tangent foot with SF-ETT decomposition 
		\begin{equation*}
			\mathcal{X} = \left\llbracket \mathcal{G}; V^{(1)}, \dotsc, V^{(d_t)}, U, \dotsc, U\right\rrbracket, \quad 
			\mathcal{G} = \mathcal{T}\left(W^{(1)}_L, \dotsc, W^{(d-1)}_L, \overline{W}^{(d)}\right),
		\end{equation*}
		then
		$P_{T_{\mathcal{X}} \mathrm{SFETT}\, (\mathbf{r}^{tt}, \mathbf{r}^{ts})} \left(Z \right) 
		\in T_{\mathcal{X}}\mathrm{SFETT}\, (\mathbf{r}^{tt}, \mathbf{r}^{ts})$ has 
		SF-ETT decomposition as follows
		\begin{align*}
			&P_{T_{\mathcal{X}}\mathrm{SFETT}\, (\mathbf{r}^{tt}, \mathbf{r}^{ts})} \left(Z \right) = \\
			&\left\llbracket \dot{\mathcal{T}}_{\mathcal{X}}\left(\dot{W}^{(1)},\dotsc,\dot{W}^{(d)}\right);
			 V^{(1)}, \dotsc, U \right\rrbracket + 
			\left\llbracket \mathcal{G};
			{\dot{V}^{(1)}}, \dotsc, 
			 U\right\rrbracket+ \dots +
			\left\llbracket \mathcal{G}; V^{(1)}, \dotsc, \dot{U} \right\rrbracket.\\
			&\dot{V}^{(i)} 
			= P_{\perp}\left(V^{(i)}\right)
			\left\llbracket Z;{V^{(1)}}^\top,\dotsc,\underbrace{I}_{i},\dotsc,{U}^\top\right\rrbracket_{(i)}
			\mathcal{G}_{(i)}^\dagger,\quad\forall i \in \overline{1,d_t},\\
			&\dot{U} = 
			P_{\perp}\left(U\right)
			\left(\sum_{i=d_t+1}^{d}
			\left\llbracket Z;{V^{(1)}}^\top,\dotsc,\underbrace{I}_{i},\dotsc,{U}^\top\right\rrbracket_{(i)}
			\mathcal{G}_{(i)}^\top\right)
			\left(\sum_{i=d_t+1}^{d}\mathcal{G}_{(i)}\mathcal{G}_{(i)}^\top\right)^{-1},\\
			&\dot{W}^{(i)} 
			= 
			P_{\perp}\left(L\left(W^{(i)}_L\right)\right)
			\left(I_{r^{\mathrm{t}}_{i}} \otimes_K
			\mathcal{G}_{ \le i - 1}^\top\right)
			\hat{\mathcal{G}}^{<i>}
			\mathcal{G}_{ \ge i + 1},\quad\forall i \in \overline{1, d_t},\\
			&\dot{W}^{(i)} 
			= 
			P_{\perp}\left(L\left(W^{(i)}_L\right)\right)\left(I_{r^{\mathrm{t}}_{s}} \otimes_K
			\mathcal{G}_{ \le i - 1}^\top\right)
			\hat{\mathcal{G}}^{<i>}
			\mathcal{G}_{ \ge i + 1},\quad\forall i \in \overline{d_t + 1, d - 1},\\
			&\dot{W}^{(d)}
			= 
			\left(I_{r^{\mathrm{t}}_{s}} \otimes_K
			\mathcal{G}_{ \le d - 1}^\top\right) 
			\hat{\mathcal{G}}^{<d>},\\ 
			&P_{\perp}\left(A\right)= \left(I - AA^\top \right), \quad \hat{\mathcal{G}} = 
			\left\llbracket Z;{V^{(1)}}^\top,\dotsc,{V^{(d_t)}}^\top,{V}^\top,\dotsc,{U}^\top\right\rrbracket.
		\end{align*}
	\end{proposition}
	\begin{proof}
		According to Proposition~\ref{state:tangent_space_structure} every 
		$P_{T_{\mathcal{X}} \mathrm{SFETT}\, (\mathbf{r}^{tt}, \mathbf{r}^{ts})} \left(Z \right) 
		\in T_{\mathcal{X}} \mathrm{SFETT}$
		has a unique representation as sum of elements of $L_0^{(i)}, L_i, L_s$
		\begin{equation*}
			P_{T_{\mathcal{X}} \mathrm{SFETT}\, (\mathbf{r}^{tt}, \mathbf{r}^{ts})} \left(Z \right)  = 
			\sum_{i=1}^{d} \eta_0^{(i)} + \sum_{i=1}^{d_t} \eta_i + \eta_s, 
		\end{equation*}
		where components can be obtained via orthogonal projectors
		\begin{equation*}
			\eta_0^{(i)} = P_{L_0^{(i)}} Z, \quad \eta_i = P_{L_i}Z, \quad \eta_s = P_{L_s}Z.
		\end{equation*}
		If $Z \in \mathbb{R}^{n_1\times \dotsc\times n_{d_t}\times n_{s}\times \dotsc\times n_s}$ then 
		then projection formulas for $\dot{V}^{(i)}$ and $\dot{U}$ are the same as in \citep{peshekhonov2024training}
		since $\mathcal{X}$ can be seen as element of $\mathrm{SFT}\, (\mathbf{r}^{ts})$.

		So one need only to find out $P_{L_0^{(i)}} Z$. Since $P_{L_0^{(i)}}$ is an orthogonal 
		projector for any $i \in \overline{1,d}$
		it satisfies
		\begin{equation*}
			\left\langle P_{L_0^{(i)}} Z, \zeta_0^{(i)}\right\rangle 
			= \left\langle Z , \zeta_0^{(i)} \right\rangle, \quad
			\forall \zeta_0^{(i)} \in L_0^{(i)}.
		\end{equation*}
		Let $i = d$ then 
		\begin{align*}
			&P_{L_0^{(d)}} Z = \left\llbracket 
				\mathcal{T}\left(W_L^{(1)}, \dots, W_L^{(d-1)}, \dot{W}^{(d)}\right); 
				V^{(1)}, \dotsc, V^{(d_t)},  U, \dotsc,  U\right\rrbracket, \\ 
			&\zeta_0^{(d)} =\left\llbracket 
				\mathcal{T}\left(W_L^{(1)}, \dots, W_L^{(d-1)}, B\right); 
				V^{(1)}, \dotsc, V^{(d_t)}, U,\dotsc,  U\right\rrbracket. 
		\end{align*}
		So the scalar product at the left hand side due to the orthogonality of Tucker factors is equal to
		\begin{align*}
			 &\left\langle P_{L_0^{(d)}} Z, \zeta_0^{(d)}\right\rangle  = 
			 \left\langle  \left\llbracket 
				\mathcal{T}\left(W_L^{(1)}, \dots, W_L^{(d-1)}, \dot{W}^{(d)}\right); 
				V^{(1)}, \dotsc, V^{(d_t)},  U, \dotsc,  U\right\rrbracket, \right.\\ 
			&\left.\left\llbracket 
				\mathcal{T}\left(W_L^{(1)}, \dots, W_L^{(d-1)}, B\right); 
				V^{(1)}, \dotsc, V^{(d_t)}, U,\dotsc,  U\right\rrbracket\right\rangle =  \\ 
			& \left\langle \mathcal{T}\left(W_L^{(1)}, \dots, W_L^{(d-1)}, \dot{W}^{(d)}\right),
			\mathcal{T}\left(W_L^{(1)}, \dots, W_L^{(d-1)}, B\right) \right\rangle.
		\end{align*}
		The scalar product at the right hand side equals to
		\begin{align*}
			 &\left\langle Z, \zeta_0^{(d)}\right\rangle  = 
			 \left\langle  Z, 
			\left\llbracket 
				\mathcal{T}\left(W_L^{(1)}, \dots, W_L^{(d-1)}, B\right); 
				V^{(1)}, \dotsc, V^{(d_t)}, U,\dotsc,  U\right\rrbracket\right\rangle =  \\ 
			& \left\langle \left\llbracket Z;{V^{(1)}}^\top, \dotsc, {V^{(d_t)}}^\top, U^\top,\dotsc,  U^\top\right\rrbracket,
			\mathcal{T}\left(W_L^{(1)}, \dots, W_L^{(d-1)}, B\right) \right\rangle.
		\end{align*}
		Thus the following identity is obtained
		\begin{align*}
			& \left\langle \mathcal{T}\left(W_L^{(1)}, \dots, W_L^{(d-1)}, \dot{W}^{(d)}\right),
			\mathcal{T}\left(W_L^{(1)}, \dots, W_L^{(d-1)}, B\right) \right\rangle = \\
			& \left\langle \left\llbracket Z;{V^{(1)}}^\top, \dotsc, {V^{(d_t)}}^\top, U^\top,\dotsc,  U^\top\right\rrbracket,
			\mathcal{T}\left(W_L^{(1)}, \dots, W_L^{(d-1)}, B\right) \right\rangle, \quad \forall B.
		\end{align*}
		Due to the  in \citep[Theorem $3.1$]{lubich2015time} one can obtain
		\begin{align*}
			&\dot{W}^{(d)} =\\
			&\left(I \otimes_K
			\mathcal{T}\left(W^{(1)}, \dots,  W^{(d)}\right)_{ \le d - 1}^\top\right) 
			\left\llbracket Z;{V^{(1)}}^\top, \dotsc, {V^{(d_t)}}^\top, U^\top,\dotsc,  U^\top\right\rrbracket^{<d>}.
		\end{align*}
		Let $i \in \overline{1,d-1}$ then 
		\begin{align*}
			&P_{L_0^{(i)}} Z = \left\llbracket 
				\mathcal{T}\left(W_L^{(1)}, \dots, \dot{W}^{(d)}, \dotsc, W_R^{(d)}, \right); 
				V^{(1)}, \dotsc, V^{(d_t)},  U, \dotsc,  U\right\rrbracket, \\ 
			&\zeta_0^{(d)} =\left\llbracket 
				\mathcal{T}\left(W_L^{(1)}, \dots, B, \dotsc, W_R^{(d)}, \right); 
				V^{(1)}, \dotsc, V^{(d_t)}, U,\dotsc,  U\right\rrbracket. 
		\end{align*}
		Left hand side scalar product is equal to
		\begin{align*}
			 &\left\langle P_{L_0^{(i)}} Z, \zeta_0^{(i)}\right\rangle  = 
			 \left\langle  \left\llbracket 
				\mathcal{T}\left(W_L^{(1)}, \dots, \dot{W}^{(d)}, \dotsc, W_R^{(d)} \right); 
				V^{(1)}, \dotsc, V^{(d_t)},  U, \dotsc,  U\right\rrbracket, \right.\\ 
			&\left.\left\llbracket 
				\mathcal{T}\left(W_L^{(1)}, \dots, B, \dotsc, W_R^{(d)} \right); 
				V^{(1)}, \dotsc, V^{(d_t)}, U,\dotsc,  U\right\rrbracket\right\rangle =  \\ 
			& \left\langle
				\mathcal{T}\left(W_L^{(1)}, \dots, \dot{W}^{(d)}, \dotsc, W_R^{(d)} \right),
				\mathcal{T}\left(W_L^{(1)}, \dots, B, \dotsc, W_R^{(d)} \right)
		\right\rangle.
		\end{align*}
		Right hand side equals to
		\begin{align*}
			 &\left\langle Z, \zeta_0^{(i)}\right\rangle  = 
			 \left\langle  Z, 
			\left\llbracket 
				\mathcal{T}\left(W_L^{(1)}, \dots, B, \dotsc, W_R^{(d)} \right); 
				V^{(1)}, \dotsc, V^{(d_t)}, U,\dotsc,  U\right\rrbracket\right\rangle =  \\ 
			& \left\langle \left\llbracket Z;{V^{(1)}}^\top, \dotsc, {V^{(d_t)}}^\top, U^\top,\dotsc,  U^\top\right\rrbracket,
				\mathcal{T}\left(W_L^{(1)}, \dots, B, \dotsc, W_R^{(d)} \right) \right\rangle. 
		\end{align*}
		Thus the following holds
		\begin{align*}
			& \left\langle
				\mathcal{T}\left(W_L^{(1)}, \dots, \dot{W}^{(d)}, \dotsc, W_R^{(d)} \right),
				\mathcal{T}\left(W_L^{(1)}, \dots, B, \dotsc, W_R^{(d)} \right)
		\right\rangle = \\
			& \left\langle \left\llbracket Z;{V^{(1)}}^\top, \dotsc, {V^{(d_t)}}^\top, U^\top,\dotsc,  U^\top\right\rrbracket,
				\mathcal{T}\left(W_L^{(1)}, \dots, B, \dotsc, W_R^{(d)} \right) \right\rangle, \quad \forall B.
		\end{align*}
		Due to the \citep[Theorem $3.1$]{lubich2015time} one can obtain
		\begin{align*}
			&\dot{W}^{(i)} =
			\left(I - L \left( W^{(i)}_L \right)L \left( W^{(i)}_L \right)^\top\right) 
			 \left(I \otimes_K
			\mathcal{T}\left(W^{(1)}, \dots, W^{(d)}\right)_{ \le i - 1}^\top\right) \cdot \\
			&\left\llbracket Z;{V^{(1)}}^\top, \dotsc, {V^{(d_t)}}^\top, U^\top,\dotsc,  U^\top\right\rrbracket^{<i>}
			\mathcal{T}\left(W^{(1)}, \dots, W^{(d)}\right)_{ \ge i + 1}.
		\end{align*}
	\end{proof}
\section{Riemannian Autodiff}\label{appendix:riemannian_autodiff}
	\begin{proposition}[Partial derivatives]\label{state:partial_derivs}
		Let $\mathcal{X} \in \mathrm{SFETT}\, (\mathbf{r}^{tt}, \mathbf{r}^{ts})$ and 
		\begin{equation*}
			\mathcal{X} = \left\llbracket \mathcal{G}; V^{(1)}, \dotsc, V^{(d_t)}, U, \dotsc, U\right\rrbracket, \quad 
			\mathcal{G} = \mathcal{T}\left(W^{(1)}_L, \dotsc, W^{(d-1)}_L, \overline{W}^{(d)}\right).
		\end{equation*}
		Let also $\mathcal{C}$ operator be a following operator 
		\begin{align*}
			&\mathcal{C}_{\mathcal{X}}\left(S^{(1)}, \dotsc, S^{(d)}; B^{(1)}, \dotsc, B^{(d_t)}; A\right) = \\
			&\left\llbracket \dot{\mathcal{T}}_{\mathcal{X}}\left(S^{(1)},\dotsc,S^{(d)}\right);
			 U^{(1)}, \dotsc, U \right\rrbracket + 
			\left\llbracket \mathcal{G};
			{B^{(1)}}, \dotsc, 
			 U\right\rrbracket+ \dots +
			\left\llbracket \mathcal{G}; U^{(1)}, \dotsc, A \right\rrbracket,\\
			& \mathcal{C}_{\mathcal{X}}(0, \dotsc, \overline{W}^{(d)}; 0, \dotsc, 0; 0 ) =\mathcal{C}_{\mathcal{X}}(\theta) = \mathcal{X}.
		\end{align*}
		Let $g$ be a scalar function defined as
		\begin{equation*}
			g = f \circ \mathcal{C}_{\mathcal{X}},
		\end{equation*}
		where $f: \mathbb{R}^{n_1 \times \dots n_{d_t} \times n_s \times \dots \times n_s} \to \mathbb{R}$ is 
		an arbitrary smooth function defined in some neighborhood of point $\mathcal{X}$. Then
		\begin{align*}
			& \frac{\partial {\left(f \circ \mathcal{C}_{\mathcal{X}} \right)}}
			{\partial {B^{(i)}}} \left( \theta \right) = 
			\left\llbracket \nabla f\left(\mathcal{X}\right);
			{V^{(1)}}^\top,\dotsc,\underbrace{I}_{i},\dotsc,{U}^\top\right\rrbracket_{(i)}
			\mathcal{G}_{(i)}^\top,\quad\forall i \in \overline{1,d_t},\\
			& \frac{\partial {\left(f \circ \mathcal{C}_{\mathcal{X}} \right)}}
			{\partial {A}} \left( \theta \right) = 
			\sum_{i=d_t+1}^{d}
			\left\llbracket \nabla f\left(\mathcal{X}\right);
			{V^{(1)}}^\top,\dotsc,\underbrace{I}_{i},\dotsc,{U}^\top\right\rrbracket_{(i)}
			\mathcal{G}_{(i)}^\top, \\
			& \frac{\partial {\left(f \circ \mathcal{C}_{\mathcal{X}} \right)}}
			{\partial {S^{(i)}}} \left( \theta \right) = 
			\left(I_{r^{\mathrm{t}}_{i}} \otimes_K
			\mathcal{G}_{ \le i - 1}^\top\right)
			\hat{\mathcal{G}}^{<i>}
			\mathcal{G}_{ \ge i + 1},\quad\forall i \in \overline{1, d_t},\\
			& \frac{\partial {\left(f \circ \mathcal{C}_{\mathcal{X}} \right)}}
			{\partial {S^{(i)}}} \left( \theta \right) = 
			\left(I_{r^{\mathrm{t}}_{s}} \otimes_K
			\mathcal{G}_{ \le i - 1}^\top\right)
			\hat{\mathcal{G}}^{<i>}
			\mathcal{G}_{ \ge i + 1},\quad\forall i \in \overline{d_{t + 1}, d_s},\\
			& \frac{\partial {\left(f \circ \mathcal{C}_{\mathcal{X}} \right)}}
			{\partial {S^{(d)}}} \left( \theta \right) = 
			\left(I_{r^{\mathrm{t}}_{s}} \otimes_K
			\mathcal{G}_{ \le i - 1}^\top\right)
			\hat{\mathcal{G}}^{<d>}, \\
			&\hat{\mathcal{G}} =
			\left\llbracket \nabla f(\mathcal{X});
			{V^{(1)}}^\top,\dotsc,{V^{(d_t)}}^\top,{U}^\top,\dotsc,{U}^\top\right\rrbracket,\\
			&\mathcal{G}= \mathcal{T}\left( W^{(1)}_L, \dotsc, W^{(d-1)}_L, \overline{W}^{(d)}\right).
		\end{align*}
	\end{proposition}
	\begin{proof}
		The partial derivatives with respect to $B^{(i)}, i \in \overline{1, d_t}$ and $A$ are 
		the same as in the case of SF-Tucker decomposition \citep{peshekhonov2024training}.
		So one only need to check the equalities above for the case 
		of partial derivatives with respect to $S^{(i)}, i \in \overline{1,d}$.

		Now  let $i \in \overline{1, d-1}$
		\begin{align*}
			&\mathrm{D}_{S^{(i)}} \left( (f \circ \mathcal{C}_{\mathcal{X}})(\theta) \right)\left[dS^{(i)}\right] = 
			\left[\mathcal{Y} =  \mathcal{C}_{\mathcal{X}} (\theta)\right] = 
			\mathrm{D}_{} \left( f(\mathcal{\mathcal{Y}}) \right)\left[
				\mathrm{D}_{S^{(i)}} \left( \mathcal{C}_{\mathcal{X}}(\theta) \right) \left[dS^{(i)}\right]  
			\right] = \\
			& \left\langle \nabla f(\mathcal{X}),
			\mathrm{D}_{S^{(i)}} \left( \mathcal{C}_{\mathcal{X}}(\theta) \right) \left[dS^{(i)}\right]\right\rangle = \\
			&\left\langle \nabla f(\mathcal{X}), \left\llbracket  
			\mathrm{D}_{S^{(i)}} \left(\dot{\mathcal{T}}_{\mathcal{X}}\left(S^{(1)}, \dotsc, S^{(d)}\right)\right)
			\left[dS^{(i)}\right]\Bigg|_{\theta}; 
			V^{(1)}, \dotsc, V^{(d_t)}, U, \dotsc, U
			\right\rrbracket\right\rangle = \\
			& \left\langle \nabla f(\mathcal{X}), \left\llbracket  
				\mathcal{T}\left(W^{(1)}_L, \dotsc, dS^{(i)}, \dotsc, W_R^{(d)}\right);
			V^{(1)}, \dotsc, V^{(d_t)}, U, \dotsc, U
			\right\rrbracket\right\rangle = \\
			&  \left\langle \left\llbracket  
				\nabla f(\mathcal{X});
			{V^{(1)}}^\top, \dotsc, {V^{(d_t)}}^\top, U^\top, \dotsc, U^\top
			\right\rrbracket,   
				\mathcal{T}\left(W^{(1)}_L, \dotsc, dS^{(i)}, \dotsc, W_R^{(d)}\right)
			\right\rangle.
		\end{align*}	
		Now applying vectorization operation on both of the vectors inside the scalar product lefts the result 
		unchanged so
		\begin{align*}
			&\mathrm{D}_{S^{(i)}} \left( (f \circ \mathcal{C}_{\mathcal{X}})(\theta) \right)\left[dS^{(i)}\right] = \\
			&  \left\langle \mathrm{vec}\,
					\underbrace{\left(\left\llbracket  \nabla f(\mathcal{X});
						{V^{(1)}}^\top, \dotsc, {V^{(d_t)}}^\top, U^\top, \dotsc, U^\top
					\right\rrbracket\right)}_{\nabla \hat{f}(\mathcal{X})},
				\mathcal{G}_{\neq i} \mathrm{vec}\, \left( dS^{(i)} \right) 
			\right\rangle = \\
			&  \left\langle \left( \mathcal{G}_{\ge i + 1} \otimes_K I \otimes_K \mathcal{G}_{\le i - 1}
				\right)^\top 
				\mathrm{vec}\,\left(\nabla \hat{f}(\mathcal{X})\right),   
				 \mathrm{vec}\, \left( dS^{(i)} \right) 
			 \right\rangle =  \\
			&\left[ \mathrm{vec}\, (AXB^\top) = (B \otimes_K A) \mathrm{vec} (X)\right] =  \\
			&\left\langle 
			\mathrm{vec}\,   
				\underbrace{\left(\left(I \otimes\mathcal{G}_{\le i - 1}^\top\right)\nabla \hat{f}(\mathcal{X})^{<i>}
				\mathcal{G}_{\ge i + 1}\right)}_{{\partial (f \circ \mathcal{C}_{{\mathcal{X}}})} / {\partial S^{(i)}} },
				 \mathrm{vec}\, \left( dS^{(i)} \right)
			\right\rangle.
		\end{align*}
		
		For $i = d$ one can perform the same operations.
	\end{proof}
		
\end{appendices}



\end{document}